\documentclass[12pt,reqno,a4paper]{amsart}
\usepackage{
    amsmath,  amsfonts, amssymb,  amsthm,   amscd,
    gensymb,  graphicx, comment,  etoolbox, url,
    booktabs, stackrel, mathtools,enumitem, mathdots,  microtype, lmodern,    mathrsfs, graphicx, tikz,  longtable,tabularx, float, tikz, pst-node, tikz-cd, multirow, tabularx, amscd,  bm, array, makecell, diagbox, booktabs,ragged2e, caption, subcaption }
\usepackage{makecell,slashbox}

\usepackage{xcolor}
\usepackage[utf8]{inputenc}
\usepackage{microtype, fullpage, wrapfig,textcomp,mathrsfs,csquotes,fbb}
\usepackage[colorlinks=true, linkcolor=blue, citecolor=blue, urlcolor=blue, breaklinks=true]{hyperref}
\usepackage[capitalise]{cleveref}
\usepackage{todonotes}
\usetikzlibrary{positioning}
\usetikzlibrary{shapes,arrows.meta,calc}
\usetikzlibrary{arrows}

\newtheorem{theorem}{Theorem}[section]

\newtheorem{corollary}[theorem] {Corollary}

\newtheorem{example}[theorem]{Example}

\newtheorem{proposition}[theorem]{Proposition}
\newtheorem{remark}[theorem]{Remark}
\newtheorem{question}[]{Question}

\newcommand\Q{\mathbb{Q}}
\newcommand\R{\mathbb{R}}
\newcommand\Z{\mathbb{Z}}
\newcommand\C{\mathbb{C}}

\newcommand{\TC}{\mathrm{TC}}
\newcommand{\Gr}{\mathrm{Gr}}
\newcommand{\ct}{\mathrm{cat}}

\newcolumntype{x}[1]{>{\centering\arraybackslash}p{#1}}

\begin{document}
\title[]{LS-category and Topological complexity of several families of fibre bundles}

\author[N. Daundkar]{Navnath Daundkar}
\address{Department of Mathematics, Indian Institute of Technology Bombay, Mumbai, 400076 India.}
\email{navnathd@iitb.ac.in}

\author[S. Sarkar]{Soumen Sarkar}
\address{Department of Mathematics, Indian Institute of Technology Madras, Chennai, 600036, India.}
\email{soumen@iitm.ac.in}

%\thanks{}

\begin{abstract}
In this paper, we study upper bounds for the topological complexity of the total spaces of some classes of fibre bundles.
We calculate a tight upper bound for the topological complexity of an $n$-dimensional Klein bottle. We also compute the exact value of the topological complexity of $3$-dimensional Klein bottle. We describe the cohomology rings of several classes of generalized projective product spaces with $\Z_2$-coefficients.  
Then we study the LS-category and topological complexity of infinite families of generalized projective product spaces. We reckon the exact value of these invariants in many specific cases. We calculate the equivariant LS-category and equivariant topological complexity of several product spaces equipped with $\Z_2$-action.

\end{abstract}
\maketitle

\section{Introduction}\label{sec:intro}
Lusternik and Schnirelmann \cite{LS-paper1} introduced a homotopy invariant of a topological space, known as the `LS-category',  to study some problems in variational calculus. This invariant has been studied widely since the 1940s, see \cite{Fox, LS-paper2, CLOT}. Two decades later, the `genus', a generalization of LS-category, of a fibration was introduced by Schwarz in \cite{Sva}. This overused term `genus' was replaced by `sectional category' in the exposition \cite{James} and subsequent articles on this topic. The notion of LS-category in the presence of group action has been studied first in \cite{eqlscategory}.

In a different context, Farber \cite{FarberTC} introduced the notion of topological complexity to study motion planning in mechanical systems. The configuration space of a mechanical system is the set of all admissible points of the system. Usually, this space has a topology. The continuous motions in the system can be characterized by continuous paths in the space. Topological complexity is a homotopy invariant and a close relative of the LS-category. Later, generalizing the concept of sectional category, Colman and Grant \cite{colmangranteqtc} introduced the equivariant sectional category. It was noted that the topological complexity is a particular case of the sectional category. However, an excellent physical interpretation of topological complexity has been mentioned in \cite{FarberTC, Farber}. 

%The notion of LS-category and topological complexity have been studied in presence of group action, see \cite{eqlscategory} and \cite{EqTCGrant} respectively.

 Let $G$ be a group and $E, B$ be two $G$-spaces such that $f\colon E \to B$ is an equivariant map. The equivariant sectional category of $f$, denoted by ${\rm secat}_G(f)$, is the least positive integer $k$ such that there exists a $G$-invariant open cover $\{V_1, \ldots, V_k\}$ of $B$ and a $G$-map $\sigma_i \colon V_i \to E$ satisfying $f \sigma_i \simeq_G \iota_{V_i} \colon V_i \hookrightarrow B$ for $i=1, \ldots, k$. If no such $k$ exists, we say ${\rm secat}_G(f)=\infty$.  If $f$ is a $G$-fibration, then one can replace the $G$-homotopy $\simeq_G$ by the equality $=$. Moreover, if $G$ is trivial then ${\rm secat}_G(f)$ is called the sectional category of $f$, and in addition, if $E$ is contractible then it becomes the LS-category of $B$, denoted by $\ct(B)$. If $X$ and $Y$ are the path connected spaces and $X \times Y$ is completely normal, then Fox \cite{Fox} showed that  
\begin{equation}\label{eq_ls_cat_prod}
   \mathrm{cat}(X\times Y)\leq \mathrm{cat}(X)+\mathrm{cat}(Y)-1. 
\end{equation}
Marzantowicz \cite{eqlscategory} introduced equivariant LS-category, denoted by $\ct_G(B)$ for a $G$-space $B$, generalizing the notion of LS-category. The $\mathrm{cat}_G(B)$ is the smallest integer $r$ such that $B$ can be covered by $r$ many $G$-invariant open subsets $U_1, \dots, U_r$ of $B$ with each inclusion $U_j\xhookrightarrow{} B$ is  $G$-homotopic to an orbit inclusion $Gb_j \hookrightarrow B$ for some $b_j \in B$ for $j=1, \ldots, r$. The sets $U_j$'s are called $G$-categorical open subsets of $B$. We note that if $b \in B^G$ and $B$ is a $G$-connected space then ${\rm secat}_G(\iota) = \ct_G(B)$ for the inclusion  map $\iota \colon \{b\} \to B$, see \cite[Corollary 4.7]{colmangranteqtc}. Also, if $G$ is trivial then $\ct_G(B)=\ct(B)$ and $U_j$'s are called categorical open subsets of $B$.

%Let $X$ be the configuration space of a mechanical system, hereafter a topological space (possibly path-connected).
Let $PB := \{\gamma ~|~ \gamma \colon [0,1]\to B ~ \mbox{is a continuous map}\}$. Consider the compact open topology on $PB$. Then $PB$ is a $G$-space defined by $(g \cdot \gamma) (t) = g\gamma(t)$,  
%Let $\gamma \colon [0,1]\to X$ be any path in $X$. 
and the fibration $$\pi_B \colon PB \to B \times B$$ defined by $\pi_{B}(\gamma)=(\gamma(0),\gamma(1))$ is a $G$-fibration. The \emph{equivariant topological complexity} of $B$, denoted by $\mathrm{TC}_G(B)$, is defined by  $\TC_G(B) = {\rm secat}_G(\pi_B)$, see \cite{colmangranteqtc}. A $G \times G$-invariant open cover $\{V_1, \dots, V_k\}$ of $B \times B$ is called a $(G \times G)$-\emph{invariant motion planning cover} for $B$ if there exists $\sigma_i \colon V_i \to PB$ satisfying $\pi_B\circ\sigma_i = \iota_{V_i} \colon V_i \hookrightarrow B \times B$ for $i=1, \ldots, k$.
%it satisfies the definition of the equivariant sectional category of $\pi_B$. 
%there exists an equivariant  continuous map $\sigma_i \colon V_{i} \to PB$ suchtha for $1\leq i \leq k$. The \emph{equivariant topological complexity} of $B$, denoted by $\mathrm{TC}_G(B)$, is the minimum of a such positive integer $k$ such that $B\times B$ admits an (equivariant) motion planning cover $ V_1, \dots, V_k$. The equivariant topological complexity is a homotopy invariant of a topological space $B$ and $\TC_G(B) = {\rm secat}_G(\pi_B)$, see \cite{EqTCGrant}.
If $G$ is trivial in ${\rm secat}_G(\pi_B)$ then it is called the topological complexity of $B$, and denoted by $\TC(B)$. Farber \cite{FarberTC} proved that the global section of $\pi_B$ cannot be continuous unless the space $B$ is contractible. 
%Topological complexity is closely related to an old invariant \emph{Lusternik-Schnirelmann category} (in short, \emph{LS-category}) of a space $X$, denoted by $\mathrm{cat}(X)$. 
The following inequalities show how LS-category and topological complexity are related. 
$$\mathrm{cat}(B)\leq \TC(B)\leq \mathrm{cat}(B\times B)\leq 2\mathrm{cat}(B)-1.$$
%The product inequality for the Lusternik-Schnirelmann category was proved in \cite{zbMATH03043823}. 
%\begin{theorem}[{\cite[Theorem 9]{zbMATH03043823}}]\label{lscprod}
%If $X$ and $Y$ are the path connected spaces. Then \[\mathrm{cat}(X\times Y)\leq \mathrm{cat}(X)+\mathrm{cat}(Y)-1.\]
%\end{theorem}
If $X$ and $Y$ are path connected spaces. Then Farber \cite[Theorem 11]{FarberTC} showed that
\begin{equation}\label{eq_tc_prod}
   \mathrm{TC}(X\times Y)\leq \mathrm{TC}(X)+\mathrm{TC}(Y)-1. 
\end{equation}
%The similar product inequality for topological complexity is proved by Farber in \cite{FarberTC}.

We note that a product space is an example of a fibre bundle, see \cite{Steenrod}. The topological complexity of fibrations has been studied in several articles such as \cite{strongeqtc, Farbergrant, Naskar}. We recall some related results here. 
Let $B$ be a $G$-space. Then $G\times G$ acts on $B\times B$ componentwise. The strongly equivariant topological complexity, denoted by  $\TC_G^*(B)$, of $B$ is the least integer $k$ for which $B\times B$ can be covered by $k$ many $(G\times G)$-invariant open sets $U_1, \dots, U_k$ such that there is a $G$-map $\sigma_i \colon U_i \to PB$ for the diagonal $G$-action on $U_i$ satisfying ${\pi \sigma_i = \iota_i \colon U_i \hookrightarrow B \times B}$ for each $1\leq i\leq k$.
%an open set $U_i$ admits a $G$-equivariant section of $\pi_B$.
If no such $k$ exists, we say $\TC_G^*(B)=\infty$. Let $G$ be a group  acting properly on a space $F$.
In \cite[Section 3]{strongeqtc}, Dranishnikov introduced the strong equivariant topological complexity $\TC_G^{\star}(F)$ of $F$ and showed that 
\begin{equation}\label{eq:Dra}
 \TC(E)\leq \TC(B)+\TC_G^{\star}(F)-1
\end{equation}
where $E$, $B$ are locally compact metric ANR-spaces and $p \colon E \to B$ is a fibre bundle with fibre $F$ having structure group $G$. On the other hand, Farber and Grant \cite[Lemma 7]{Farbergrant} showed that
\begin{equation}\label{eq:FG}
\TC(E)\leq \TC(F)\cdot \ct(B\times B)
\end{equation}
where $p \colon E \to B$ is a Hurewicz fibration with fibre $F$ and base $B$.
Later, Grant \cite[Theorem 3.1]{Grantfibrations} improved the upper bound in the above inequality. 
Therefore, it is natural to ask the following.
\begin{question}\label{que}
Let $p \colon E \to B$ be a fibre bundle with fibre $F$. When can we say that $$ \mathrm{TC}(E)\leq \mathrm{TC}(B)+\mathrm{TC}(F)-1?$$
%Under what conditions on the fibre bundle, do we have an additive upper bound on the topological complexity of the total space of fibre bundles in terms of the topological complexities of corresponding fibre and base space?  
\end{question}
In this paper, we answer \Cref{que}  for a class of fibre bundles, see \Cref{thm:TC-fibrebundle}. Then, as a consequence, we compute the tight upper bounds for the topological complexity of a class of generalized projective product spaces and Dold manifolds \cite{sarkargpps}, including an $n$-dimensional Klein bottle introduced by Davis in \cite{DDavis}. We also study the LS-category and topological complexity of  several generalized projective product spaces which includes `Dold manifolds of Grassmann type'. We note that \Cref{thm:TC-fibrebundle} gives much better upper bounds than \eqref{eq:Dra} and \eqref{eq:FG} for many fibre bundles, see Example \ref{rem_dra}, \cref{prop:TCrefaction}, and \Cref{thm:TC-cpn-pn}.

We recall the definition of generalized projective product spaces following \cite{sarkargpps}. Let $M$ and $N$ be manifolds with involutions $\tau \colon M \to M$ and $\sigma \colon N \to N$ such that $\sigma$ is fixed point free. Define the  identification space:
\begin{equation}\label{eq_gen_dman}
X(M, N) :=\displaystyle\frac{M \times N}{(x,y)\sim (\tau(x), \sigma(y))}.
\end{equation}
Then $X(M, N)$ is a manifold of dimension $\dim(M) + \dim(N)$ and there is a fibre bundle 
\begin{equation}\label{eq:indc_fib_bndl}
    M \xhookrightarrow{} X(M, N) \stackrel{\mathfrak{p}}\longrightarrow N/\left<\sigma \right>
\end{equation}
 defined by $\mathfrak{p}([(x,y)])=[y]$, where $N/\left<\sigma \right>$ is the orbit space of the group $\left<\sigma \right>$ action on $N$ induced by the involution $\sigma$. Note that this class of manifolds contains all projective product spaces \cite{Davis} and Dold manifolds \cite{Dold}. Note that $M, N$ in \eqref{eq:indc_fib_bndl} can be replaced by the topological spaces with similar involutions. However, in this paper, we are interested in the manifold category, unless it is mentioned. 
 
%In this paper, we give an upper bound on the topological complexity of these spaces which include projective product spaces, generalized Dold manifolds and `Dold manifolds of Grassmann type'. In many cases we compute the exact value.  We also study their equivariant topological complexity.

This article is organized as follows. 
In \Cref{sec: TC of fb}, we give a counter-example of \cite[Theorem 3.4]{Naskar}. Then, we obtain an upper bound for the topological complexity of the total space of a class of fibre bundles as the sum of the topological complexities of the base and the fibre, see Theorem \ref{thm:TC-fibrebundle}. We recall the result \cite[Theorem 2.6]{Naskar} which is used in the later sections to compute LS-category of several fibre bundles. We discuss some applications.

In \Cref{sec:TCKn}, we recall the definition of an $n$-dimensional Klein bottle $K_n$ for $2 \leq n \in \Z$. Then we show that $\TC (K_3)=6$. For higher $n$, we calculate (possibly) almost tight upper bounds on the topological complexity of $K_n$. 
%Here, the computation $\TC (K_2)=5$ is an alternative proof for the topological complexity of classical $2$-dimesnional Klein bottle. 

We begin \Cref{sec: Cat TC some gpps} by recalling the definition of projective product spaces and some generalized projective product spaces. Then we compute the mod-$2$ cohomology ring of certain generalized projective product spaces, see \Cref{thm: cohoringYM}. This result extends some previous results in \cite{Davis} and \cite{sarkargpps}.
%Using the cohomology ring structure in \Cref{thm: cohoringYM} and \Cref{thm:TC-fibrebundle},
Then, we compute bounds (probably tight) as well as exact values for the LS-category and topological complexity of an infinite families of generalized projective product spaces. 
%Then we compute bounds on their equivariant LS-category and equivariant topological complexity in \Cref{sec: eq Cat TC some gen Dold}.

In \Cref{sec: Cat TC some gen Dold}, we introduce a new class of generalized projective product spaces called `Dold manifolds of Grassmann type' whose topological properties have not been studied till now. Then we describe their cohomology ring and compute their LS-category, see \Cref{thm_lscat_gras}. We discuss some lower and upper bounds for the topological complexity of these spaces, see  \Cref{prop_tc_grdn} and \Cref{thm:TC-cpn-pn}. 

In \Cref{sec: eq Cat TC some gen Dold}, we study the equivariant LS-category and equivariant topological complexity of several $\Z_2$-spaces related to the generalized projective product spaces and generalized Dold manifolds. We compute them for several cases, see \Cref{prop:eqi_lstc} and \Cref{ex:eqi_lstc}.

%We use this result to give some tight upper bounds on the topological complexity of generalized projective product space  For $2$ and $3$-dimensional Klein bottle we compute the exact value. 
%{\color{red} We need to write this part in details}. 

\section{LS-category and topological complexity of some fibre bundles}\label{sec: TC of fb}
An upper bound for the topological complexity of the total space of a fibre bundle has been studied in \cite{Farbergrant, Grantfibrations} as the product of the topological complexity of the base and the fibre and in \cite{Naskar} as the sum  of the topological complexity of the base and the fibre (under some hypotheses). The authors of this paper found that the arguments in the proof of \cite[Theorem 3.4]{Naskar} is inadequate. However, importantly, this observation does not affect the main goals (computing LS-category and topological complexity of Dold manifolds) of the paper \cite{Naskar}. 
We begin this section with a counter-example to \cite[Theorem 3.4]{Naskar} observed by the first author of this paper. Then, we study  upper bounds for the topological complexity of the total spaces of a class of fibre bundles. We compute an upper bound for the LS-category and topological complexity of manifolds defined in \eqref{eq_gen_dman}.

%\begin{theorem}[{\cite[Theorem 3.4]{Naskar}\label{nsk1}}]
%Let $F$, $E$ and $B$ be path-connected spaces and $E\stackrel{p}{\longrightarrow} B$ be a fibre bundle with fibre $F$ and $V_1,\dots,V_m$ be an open cover of $B\times B$ with homotopy sections $\sigma_j: V_j\to PR_j\subseteq PB$ of $\pi:PB\to B\times B$ such that over $R_j$ the bundle $E\stackrel{p}{\longrightarrow} B$ is trivial for $j=1,\dots,m$. Let \[h_j:V_j\times(F\times F)\to (p\times p)^{-1}(V_j)\] be a local trivialization for the bundle $E\times E\stackrel{p\times p}{\longrightarrow} B\times B$ with fibre $F\times F$ for $j=1,\dots,m$. Then \[TC(E)\leq TC(F)+m-1.\]
%\end{theorem}

\begin{example}[{Counter-example to \cite[Theorem 3.4]{Naskar}}]\label{ex:counter}
Let $K_2$ be the classical $2$-dimensional Klein bottle. Note that $K_2$ can be considered as the total space of a fibre bundle over $S^1$ with fibre $S^1$, see \cite[Chapter 1]{Steenrod}. We now construct an open cover for $S^1\times S^1$ with three open sets satisfying the hypothesis of \cite[Theorem 3.4]{Naskar}. Let $a_i\in S^1$ for $1\leq i\leq 3$ be three distinct points on the circle. We write $a_i$ instead $\{a_i\}$ for short. 
Suppose that $$R_i:=S^1\setminus a_i, \mbox{ and } V_i:=R_i\times R_i, \mbox{ for } 1\leq i\leq 3  .$$ Note that $S^1= \cup_{i=1}^{3}R_i$ and $S^1\times S^1= \cup_{i=1}^{3}V_i$.
Let $(x,y)\in V_i$. Then $x\neq a_i$ and $y\neq a_i$. Therefore there is a unique 
%( {\color {red} or geodesic path? Here geodesic path will not give continuous section since if $x$ is very close to $a_i$ from one side and if $y$ is very close to $a_i$ from other side, then the geodesic wont  lie inside $R_i$ as it will contain $a_i$}.) 
geodesic $\gamma_{(x,y)}^i$ from $x$ to $y$ which does not contain $a_i$. Thus the path $\gamma_{(x,y)}^i$ lies inside $R_i$.
Therefore, for $1\leq i\leq 3$, we can define continuous sections $$\sigma_i \colon V_i \to PS^1 \mbox{ by } \sigma_i(x,y):=\gamma_{(x,y)}^i$$ of $\pi_{S^1} \colon PS^1\to S^1\times S^1$.  
Note that $\sigma_i(V_i)\subseteq PR_i$. Since each $V_j$ is contractible, we have local trivialization
\[h_j \colon V_j\times(S^1\times S^1)\to (p\times p)^{-1}(V_j)\]  for the bundle $K_2\times K_2\stackrel{p\times p}{\longrightarrow} S^1\times S^1$ with fibre $S^1\times S^1$ for $j=1,2,3$. 

The above discussion shows that the hypotheses of \cite[Theorem 3.4]{Naskar} are satisfied for this fibre bundle. From the conclusion of \cite[Theorem 3.4]{Naskar}, we get $\TC(K_2)\leq 2+3-1=4 $. Which is not possible, since $\TC(K_2)=5$, see \cite{TCKleinbottle}. 
\end{example}

% This example was observed by the first author and motivated the authors of this paper to achieve the following.
 %study upper bounds for the topological complexity of the total spaces of fibre bundles. Then we use the result to compute several tight upper bound on the topological complexity of generalized projective product spaces. 

\begin{theorem}\label{thm:TC-fibrebundle}   Let $F\xhookrightarrow {}E\stackrel{p}{\longrightarrow} B$ be a fibre bundle where $F\times F$, $E \times E$ and $B \times B$ are completely normal spaces. Let $\{U_1,\dots,U_m\}$ and $\{V_1,\dots,V_n\}$ be motion planning covers of $B\times B$ and $F\times F$ with sections $s_i:U_i\to PB$ and $s'_j:V_j\to PF$ of $\pi_B$ and $\pi_F$, respectively. 
Suppose the following conditions hold.
\begin{enumerate}
    \item There is an open cover $\{R_1,\dots, R_m\}$ of $B$ which gives local trivialization of $F\xhookrightarrow {}E\stackrel{p}{\longrightarrow} B$ such that $s_i:U_i\to PR_i\subseteq PB$.
    \item There are local trivialization's $h_i \colon (p\times p)^{-1}(\overline{U_i}) \to \overline{U_i} \times F \times F$ of the product fibre bundle $F\times F\xhookrightarrow {}E\times E\stackrel{p\times p}{\longrightarrow} B\times B$ such that $(\overline{U_i}\cap \overline{U_j})\times V_k$ is invariant under $h_jh_i^{-1}$ for all $1\leq i \neq j\leq m$ and $1\leq k\leq n$.
\end{enumerate}
 Then $\TC(E)\leq m+n-1$.
\end{theorem}

%{\color{red}\begin{theorem}\label{thm:TC-fibrebundle}    Let $F\xhookrightarrow {}E\stackrel{p}{\longrightarrow} B$ be a fibre bundle where $F\times F$, $E \times E$ and $B \times B$ are completely normal spaces. Let $\{U_1,\dots,U_m\}$ and $\{V_1,\dots,V_n\}$ be motion planning covers of $B\times B$ and $F\times F$, respectively such that there is a section $\sigma_i \colon U_i \rightarrow PR_i \subseteq PB$ of $\pi_B$ such that over $R_i$ the bundle $p$ is trivial for $j=1,...,m$.. Let \[h_i \colon (p\times p)^{-1}(\overline{U_i}) \to \overline{U_i} \times F \times F\] be a local trivialization of $F\times F\xhookrightarrow {}E\times E\stackrel{p\times p}{\longrightarrow} B\times B$ such that $(\overline{U_i}\cap \overline{U_j})\times V_k$ is invariant under $h_jh_i^{-1}$ for all $1\leq i<j\leq m$ and $1\leq k\leq n$. Then $\TC(E)\leq m+n-1$.\end{theorem}}
\begin{proof}
Let $X_i:=U_1\cup \dots \cup U_i$ and $Y_j:=V_1\cup \dots \cup V_j$ for $1\leq i \leq m$ and $1\leq j\leq n$. So, we have the following. $$\emptyset=X_0\subseteq X_1\subseteq \dots \subseteq X_m=B\times B \mbox{ and } \emptyset=Y_0\subseteq Y_1\subseteq \dots \subseteq Y_n=F\times F.$$  Define  $s \colon B\times B \to PB$ and $s'\colon F\times F\to PF$ such that  $s|_{X_{i+1}-X_{i}}=s_i$ and $s'|_{Y_{j+1}-Y_{j}}=s'_j$ for $1\leq i\leq m$ and $1\leq j\leq n$. Then, $s \colon B\times B \to PB$ and $s'\colon F\times F\to PF$ are sections (as maps) of $\pi_B$ and $\pi_F$ respectively, satisfying \cite[Proposition 4.12]{Farber}. 

%with $$\emptyset=X_0\subseteq X_1\subseteq \dots \subseteq X_m=B\times B \mbox{ and } \emptyset=Y_0\subseteq Y_1\subseteq \dots \subseteq Y_n=F\times F $$ such that $s|_{X_{i+1}-X_{i}}$ and $ s'|_{Y_{j+1}-Y_{j}}$ are continuous for all $i=0,1,\dots,m-1$ and $j=0,1,\dots,n-1$. Here one can observe that, $s(b_1,b_2)=s_i(b_1,b_2)$ where $i$ is the smallest index such that $(b_1,b_2)\in U_i$. Similarly $s':F\times F\to PF$ is defined. Also note that $s|_{X_{i+1}-X_{i}}=s_i$ and $s'|_{Y_{j+1}-Y_{j}}=s'_j$ for $1\leq i\leq m$ and $1\leq j\leq n$.
 
Now we construct a finite sequence of open sets which converges to $E\times E$ using $U_i$'s and $V_j$'s. Let $C_0=\emptyset$, $C_1 :=h_1^{-1}(U_1 \times V_1)$, and
$$C_\ell := C_{\ell -1} \cup h_\ell^{-1}(U_\ell \times V_1) \cup h_{\ell - 1}^{-1}(U_{\ell -1} \times V_2) \times \cdots \times h_2^{-1}(U_2 \times V_{\ell -1}) \cup h_1^{-1}(U_1 \times V_\ell)$$  for each $\ell \in \{2,\dots, m+n-1\}$. Then each $C_\ell$ is open in $E \times E$. Also, using $X_i$'s and $Y_j$'s, we define
\[Q_\ell:=Q_{\ell-1}\bigcup \bigg(\bigcup_{r=1}^{\ell}h_{r}^{-1}((X_r-X_{r-1})\times (Y_{\ell-r+1}-Y_{\ell-r})) \bigg),\]
 for each $\ell \in \{1,\dots, m+n-1\}$.
Following the arguments given in the proof of \cite[Theorem 2.6]{Naskar}, one can show that $C_\ell = Q_\ell$ for $\ell \in \{1,\dots, m+n-1\}$. Therefore, each $Q_\ell$ is open in $E \times E$.  
Observe that $\emptyset=Q_0\subseteq Q_1\subseteq \dots \subseteq Q_{m+n-1}=E\times E.$
We note that $X_i$'s and $Y_j$'s are empty sets if the corresponding index of $X$ and $Y$ exceeds $m$ and $n$, respectively.   
%Let $Q_0=\emptyset$, $Q_1=h_1^{-1}(X_1\times Y_1)$, $Q_2=Q_1\cup h_2^{-1}((X_2-X_1)\times Y_1)\cup h_1^{-1}(X_1\times (Y_2-Y_1))$ $\dots$ ,$Q_{m+n-1}=Q_{m+n-2}\cup h_{m}^{-1}((X_m-X_{m-1})\times (Y_n-Y_{n-1}))=E\times E$. 
Observe that 
\begin{equation}\label{disjunion}
    Q_{\ell+1}-Q_\ell= \bigcup_{r=1}^{\ell+1}h_r^{-1}((X_r-X_{r-1})\times (Y_{\ell+2-r}-Y_{\ell+1-r})).
\end{equation}
 We claim that the right hand side of \eqref{disjunion} is  a disjoint union. 
For each $\ell \in \{0,1,\dots, m+n-1\}$, we write $A^\ell_r:=(X_r-X_{r-1})\times (Y_{\ell+2-r}-Y_{\ell+1-r})$. We now show that $h_r^{-1}(A_r^\ell)\cap \overline{h_{r'}^{-1}(A_{r'}^\ell})=\emptyset$  for $1\leq r<r'\leq \ell+1$. On contrary, suppose $x \in h_r^{-1}(A_r^\ell)\cap \overline{h_{r'}^{-1}(A_{r'}^\ell})$. 
So $$h_{r'}(x) \in h_{r'}(h_r^{-1}(A_r^\ell) \cap (p \times p)^{-1}(\overline{U_{r'}})) \subseteq h_{r'}(h_r^{-1}(((X_r - X_{r-1})\cap \overline{U_{r'}}) \times V_{\ell +2 -r})).$$
Since $h_{r'} h_r^{-1}((\overline{U_r} \cap \overline{U_{r'}}) \times V_{k})= (\overline{U_r} \cap \overline{U_{r'}}) \times V_{k}$ and $p \times p$ is a fibre bundle, we have $$h_{r'}(x) := (a_{r'}, b_{r'}) \in ((X_r - X_{r-1}) \times V_{\ell+2-r}.$$

Now consider $x \in \overline{h_{r'}^{-1}(A^\ell_{r'})}$. We have  $\overline{h_r^{-1}(A^\ell_{r'})}=h_{r}^{-1}(\overline{A^\ell_{r'}})$ from the hypothesis on $h_r$. Thus, $h_{r'}(x)\in  \overline{A_{r'}^\ell}$. 
Therefore, $$(a_{r'},b_{r'})\in \overline{A_{r'}^\ell}\subseteq \overline{(X_{r'}-X_{r'-1})}\times \overline{(Y_{\ell+2-r'}-Y_{\ell+1-r'})}.$$ Therefore, we have a sequence $a_{r'}(n)$ in $(X_{r'}-X_{r'-1})$ that converges to $a_{r'}$. Since $X_{r'-1}^c$ is closed, $a_{r'}\in X_{r'-1}^c$. This is a contradiction to the fact that $a_{r'}\in X_{r}-X_{r-1}\subseteq X_{r'-1}$. Similarly, we can show that $\overline{h_r^{-1}(A_r^\ell)}\cap h_{r'}^{-1}(A_{r'}^\ell)=\emptyset$  for $1\leq r<r'\leq \ell+1$. 
Therefore, complete normality shows that the right hand side of \eqref{disjunion} is a disjoint union. 

Now we explain here general phenomenon from the proof of $part ~~(b)\implies part ~~ (c)$  of  \cite[Proposition 4.12]{Farber}.
For $0\leq \ell\leq m+n-1 $ consider the sets \[W_{\ell+1}=Q_{\ell+1}-(Q_1\cup \dots\cup Q_{\ell}).\] 
Observe that $E\times E\subseteq \cup_{\ell=1}^{m+n-1} W_{\ell}$. 
Therefore, we need to define a section $s_E$ of $\pi_E$ such that the restriction of $s_E$ on each $W_{\ell}$ is continuous for $0\leq \ell\leq m+n-1$. Since $W_{\ell+1}\subseteq Q_{\ell +1}-Q_{\ell}$, it is enough to defined a continuous section of $\pi_E$ on each $Q_{\ell+1}-Q_{\ell}$. Recall that \eqref{disjunion} is in fact a disjoint union 
\[ Q_{\ell+1}-Q_{\ell}= \bigsqcup_{r=1}^{\ell+1}h_r^{-1}((X_r-X_{r-1})\times (Y_{\ell+2-r}-Y_{\ell+1-r}))=\bigsqcup_{r=1}^{\ell+1}h_r^{-1}(A_r^{\ell}).\] 
Therefore, it is enough to define section of $\pi_E$ which is continuous on each $h_r^{-1}(A_r^{\ell})$.
Now recall that, $s_i|_{X_{i+1}-X_i}$ and $s'_j|_{Y_{j+1}-Y_{j}}$ are continuous for $1\leq i\leq m$ and $1\leq j \leq n$. Therefore, we can define a
continuous section of $\pi_E$ on $h_r^{-1}(A_r^{\ell})$ as the compositions of following maps 
\[h_r^{-1}(A_r^{\ell}) \stackrel{h_r}\longrightarrow U_r\times V_{\ell+2-r}\stackrel{s_r\times s'_{\ell+2-r}}\longrightarrow PR_r\times PF\longrightarrow P(R_r\times F)\to PE. \]
%as $h_r^{-1}\circ (s_r,s'_{\ell+2-r})$ on $h_r^{-1}((X_r-X_{r-1})\times (Y_{\ell+2-r}-Y_{\ell+1-r}))$ 
for each $1\leq r\leq \ell+1$. Consequently,  we have a continuous section of $\pi_E$
on $Q_{\ell+1}-Q_{\ell}$ for $1\leq \ell \leq m+n-1$. This proves the theorem.
\end{proof}

\begin{remark}
 One can omit the closure condition on $U_i$ in the map $h_i$ if the spaces are regular. Also, if $E = F \times B$, then we get \cite[Theorem 11]{FarberTC}. Also, Example \ref{ex:counter} shows that one cannot relax condition (2) in Theorem \ref{thm:TC-fibrebundle} to achieve the conclusion.    
\end{remark}

Naskar and the second author \cite{Naskar} discussed an upper bound for the LS-category of the total spaces of certain fibre bundles. We recall that result in the following.

\begin{theorem}[{\cite[Theorem 2.6]{Naskar}}]\label{thm:cat}
Let $F\xhookrightarrow{} E\stackrel{p}{\longrightarrow} B$ be a fibre bundle, where  $F$, $E$ and $B$ are path-connected, Hausdorﬀ, second countable topological spaces. Suppose that $E$ is completely normal. 
Let $\{U_1,\dots,U_{m} \}$ be a categorical cover of $B$
such that $\phi_i \colon \overline{p^{-1}(U_i)}\to \overline{U_i}\times F$ is a homeomorphism for $1\leq i\leq m$. Moreover, there is a categorical cover  $\{V_1,\dots,V_{n}\}$ of $F$ such that $(U_{i'}\cap U_i)\times V_j$ are invariant under $\phi_{i'}\phi_{i}^{-1}$ for all $\{i,i'\}\subseteq \{1,\dots,m\}$ and $j\in\{1,\dots,n\}$. Then $\mathrm{cat}(E)\leq \mathrm{cat}(F)+\mathrm{cat}(B)-1$.
\end{theorem}

Let $\left<\tau \right>$ be the group generated by the involution $\tau$ on $M$. So, $\left<\tau \right>$ acts on $M$ via $\tau$.  The following result is a consequence of \Cref{thm:cat}.
\begin{proposition}\label{prop cat}
Let $X(M, N)$ be a generalized projective product space as define in \eqref{eq_gen_dman}. Let $\{V_1,\dots, V_q\}$ be a $\left<\tau \right>$-invariant categorical cover of $M$. Then \[{\mathrm{cat}(X(M,N)) \leq q +\mathrm{cat}(N/\left<\sigma \right>)-1}.\]
\end{proposition}
\begin{proof}
Let $\ct(N/\left<\sigma \right>)=r$ and $\{U_1,\dots,U_r\}$ be a categorical cover of $N/\left<\sigma \right>$. The orbit map ${\pi \colon N \to N/\left<\sigma \right>}$  is a double cover, as $\sigma$ is fixed point free. Therefore, for $1\leq i\leq r$, we have the following. \[p^{-1}(U_i)=\frac{M\times \pi^{-1}(U_i) }{(x,y)\sim (\tau(x),\sigma(y))}\cong U_i\times M.\] 
Let $\phi_i \colon p^{-1}(U_i)\to U_i\times M$ be this homeomorphism. So, it is a local trivialization of ${M \xhookrightarrow{} X(M, N) \stackrel{\mathfrak{p}}\longrightarrow N/\left<\sigma \right>}$ for $i=1, \ldots, r$. Now $\phi_{i}^{-1}(([y],x))=[(x,y)]\in M\times \pi^{-1}(U_i)/\sim$ if $[y] \in U_i$. Since the structure group for the bundle ${M \hookrightarrow X(M, N) \xrightarrow{p} N/\left<\sigma \right>}$ is $\Z_2 \cong \left < \tau \right>$, then either
\[\phi_{j}\circ\phi_{i}^{-1}(([y],x))= 
    ([y],x) \mbox{ or } \phi_{j}\circ\phi_{i}^{-1}(([y],x)) = ([y],\tau(x)).\]
Thus, $(U_i\cap U_j)\times V_k$ is invariant under  $\phi_{j}\circ\phi_{i}^{-1}$ for all $1\leq i,j\leq r$ and $1\leq k\leq q$. Then, the proposition follows from \Cref{thm:cat}, since $M$, $N$ and $X(M,N)$ are manifolds.
\end{proof}
%The local trivialization $\phi_{i}:p^{-1}(U_i)\to U_i\times M$ is defined as 
%\[\phi_{i}([x,y])=\begin{cases}
  %  ([y],\tau(x)) & y\in \tilde{U}_i\\
  %  ([y],x) & y\in \sigma(\tilde{U}_i).
%\end{cases}\] 

%The following result is a straightforward application of \Cref{thm:TC-fibrebundle}.
Next, we have the following as an application of \Cref{thm:TC-fibrebundle}.
\begin{proposition}\label{cor:TC of gpps}
Let $X(M, N)$ be a generalized projective product space as define in \eqref{eq_gen_dman}. Let $\{V_1',\dots,V_q'\}$ be a $(\left<\tau \right>\times \left< \tau \right>)$-invariant motion planning cover of $M$.
%open cover of $M \times M$ such that there exist continuous local sections $V_i\to PM$ for each $1\leq i\leq q$.
Then \[\TC(M)\leq q + \mathrm{cat}(N/\left<\sigma \right> \times N/\left<\sigma \right>)-1.\] 
%In particular, \[\TC(M)\leq \mathrm{cat}_{\left<\tau \right> \times \left< \tau \right>}(M \times )+\mathrm{cat}(N/\sigma \times N/\sigma)-1.\] %where $\TC_{\left<\tau \right>}$ is the $\left<\tau \right>$-equivariant topological complexity of $M$. 
\end{proposition}
\begin{proof}
%The proof is somewhat similar to the proof of \Cref{prop cat}, and we briefly write the arguments. 
Let $U_1', \ldots, U_m'$ be a categorical cover of $N/\left<\sigma \right> \times N/\left<\sigma \right>$. So, each $U_i'$ is a subset of a contractible subset of $N/\left<\sigma \right> \times N/\left<\sigma \right>$ and $U_1', \ldots, U_m'$ is a motion planning cover of $N/\left<\sigma \right> \times N/\left<\sigma \right>$. Therefore, $U_i'$'s satisfy (1) in Theorem \ref{thm:TC-fibrebundle}, and for $1\leq i\leq m$, we have the following. 
\[(p \times p)^{-1}(U_i)=\frac{M \times M \times (\pi \times \pi)^{-1}(U_i') }{(x_1,y_1,x_2,y_2)\sim_p (g\cdot (x_1, y_1),h\cdot( x_2,y_2))} \cong U_i' \times (M \times M),\]
where $g,h\in \{(\tau, \sigma), (\rm{Id}, \rm{Id})\}$ and $(\pi \times \pi)^{-1}(U_i') \subseteq N \times N$.  
We denote this homeomorphism by $h_i \colon (p \times p)^{-1}(U_i')\to U_i'\times (M \times M)$. So, it is a local trivialization of
\begin{equation}\label{eq_prod_bund}
M \times M \xhookrightarrow{} X(M, N) \times X(M, N) \stackrel{\mathfrak{p} \times \mathfrak{p}}\longrightarrow N/\left<\sigma \right> \times N/\left<\sigma \right>
\end{equation}
for $i=1, \ldots, m$. Now $h_{i}^{-1}(([y_1, y_2], x_1, x_2))=[(x_1, x_2, y_1, y_2)]\in M \times M \times (\pi \times \pi)^{-1}(U_i')/\sim_p$ if $[y_1, y_2] \in U_i'$. Since the structure group for the bundle \eqref{eq_prod_bund} is $\Z_2^2 \cong \left < \tau \right>^2$, then
\[h_{j}\circ h_{i}^{-1}(([y_1, y_2], x_1, x_2)) = ([y_1, y_2],\tau_1(x_1), \tau_2(x_2))\] for some $\tau_1, \tau_2 \in \left< \tau \right>$. 
Thus, the set $(U_i'\cap U_j') \times V_k'$ is invariant under  $h_{j}\circ h_{i}^{-1}$ for all $1\leq i, j\leq m$ and $1\leq k\leq q$.
%$ one can show that any $\left<f \right>\times \left<f \right>$-invariant motion planning cover of $M$ satisfies the second condition in
Note that $M, N$ are regular, since they are manifolds. 
Thus $U_i'$'s and $V_\ell'$'s satisfy the hypotheses of \Cref{thm:TC-fibrebundle}. Hence, the conclusion follows.
\end{proof}
\begin{remark}
 %   \begin{enumerate}
  %      \item The conclusion of \Cref{thm:TC-fibrebundle} and \cite[Theorem 3.4]{Naskar} is same. However, one extra hypothesis is needed in Theorem \ref{thm:TC-fibrebundle}.
 The arguments of the proof of \Cref{cor:TC of gpps} and \Cref{prop cat} can be applied if $M, N$ are metric spaces.  
  %  \end{enumerate}
\end{remark}

\begin{proposition}
Let $X(M, N)$ be a generalized projective product space. Then $$\TC (X(M,N)) \leq \TC (N/\left<\sigma \right>) + \TC_{\left< \tau \right>}^*(M)-1.$$ 
\end{proposition}
\begin{proof}
Recall from \eqref{eq:indc_fib_bndl} that we have a fibre bundle $M\xhookrightarrow{}X(M,N)\to N/\left<\sigma \right>$ whose structure group is $\left< \tau\right>$.
Also recall that an action of a finite group on a Hausdorff space is always proper.   
Therefore, $\left< \tau\right>$ acts properly on $M$.
Thus it follows from \cite[Theorem 3.1]{strongeqtc} that \[\TC (X(M,N)) \leq \TC (N/\left<\sigma \right>) + \TC_{\left< \tau \right>}^*(M)-1.        \qedhere\]
\end{proof}

%\begin{remark}\label{rem_dra}
%Observe that there are similarities in the bounds of topological complexity given in \Cref{cor:TC of gpps} and \cite[Theorem 3.1]{strongeqtc}. Here we mention an example where the bound given by   \cite[Theorem 3.1]{strongeqtc} may not help us to compute the topological complexity of the total space of some fibre bundle. On the other hand, bounds given in \Cref{cor:TC of gpps} indeed help us to get the tight bounds on the total space of the fibre bundle.
\begin{example}\label{rem_dra}
Consider the following description of the Klein bottle 
\[K\cong \frac{S^1\times S^1}{(z_1,z_2)\sim (\bar{z}_1,-z_2)}.\] \
Then we have a fibre bundle
 \[S^1\xhookrightarrow{}K\to \R P^1.\]
 Here the structure group of this $S^1$-bundle is $\Z_2$ generated by the involution $z\to \bar{z}$. 
 Then it is well known that $ \infty = \TC_{\Z_2}(S^1) \leq \TC_{\Z_2}^*(S^1)$. 
 Therefore, the bounds described in \cite[Theorem 3.1]{strongeqtc} does not give us the good upper bound for $\TC(K)$.
 On the other hand, it is known and also we have shown in the next section in \Cref{thm: tc K} that the upper bound described in \Cref{cor:TC of gpps} gives us an upper bound for $\TC(K)$ to be $5$.
 Therefore, this example shows that the upper bound in \Cref{cor:TC of gpps} is stronger than the upper bound given by \cite[Theorem 3.1]{strongeqtc}.
\end{example}

\section{Topological complexity of an $n$-dimensional Klein bottle}\label{sec:TCKn}
Consider $S^1=\{z \in \C ~ |~ z\bar{z}=1\}$ and the $n$-dimensional torus $(S^1)^n$ and an identification $\sim$ on $(S^1)^n$ defined by $(z_{1},\dots,z_{n-1},z_{n})\sim( \bar{z}_{1},\dots,\bar{z}_{n-1},-z_{n})$. Then the identification space $K_n :=(S^1)^n/\sim$ is called an $n$-dimensional Klein bottle in \cite{DDavis}. 
%\[K_{n}:=\frac{(S^{1})^{n}}{(z_{1},\dots,z_{n-1},z_{n})\sim( \bar{z}_{1},\dots,\bar{z}_{n-1},-z_{n})}.\] 
Observe that $K_n$ is a generalized projective product space, where $\tau$ is the involution on $M=(S^1)^{n-1}$ generated by conjugation on each component and $\sigma$ is the antipodal involution on $S^1$.
%Davis determined  the integral cohomology algebra, stable homotopy type of $K_n$. He also computed zero divisor cup length and an explicit immersion and embedding in Euclidean space.
In this section, we compute an upper bound for $\TC(K_n)$ and show that this bound coincides with  $\TC(K_n)$ for $n= 3$. %Therefore, in some sense, we give an different proof for the computation of the topological complexity of the classical $2$-dimensional Klein bottle.

\begin{theorem}\label{TCKn}\label{thm: tc K}
We have the following bounds for $\TC(K_n)$.
  \[n+3\leq \TC(K_n)\leq 
\begin{cases}
3k+3 & \text{ if } $n=2k+1$\\
3k+2& \text{ if } $n=2k$.
\end{cases}\]  In particular,  $\TC(K_2)=5$ and $\TC(K_3)=6$.
\end{theorem}
\begin{proof}
Note that we have a fibre bundle \[(S^1)^{n-1}\xhookrightarrow{} K_n\xrightarrow{\mathfrak{p}_n} S^1/\sigma =\mathbb{R}P^1\] as in \eqref{eq:indc_fib_bndl} for each $n \geq 2$, and it induces a product bundle \[(S^1)^{n-1}\times (S^1)^{n-1} \xhookrightarrow{} K_n\times K_n \xrightarrow{\mathfrak{p}_n \times \mathfrak{p}_n}  \mathbb{R}P^1\times \R P^1.\] 
%We use \Cref{cor:TC of gpps} to find the upper bound on $\TC(K_n)$. In particular, we have to find $(\Z_{2}\times \Z_2)$-invariant open cover of fibre $(S^1)^{n-1}\times (S^1)^{n-1}$ such that each open set admits a local continuous section of the free path space fibration \[\pi_{(S^1)^{n-1}}:P(S^1)^{n-1}\to (S^1)^{n-1}\times (S^1)^{n-1}.\]
Recall that the action of $\Z_{2}$ on $(S^1)^{n-1}$ is defined by the conjugation on each component. So, $(\Z_{2}\times \Z_2)$ acts on $(S^1)^{n-1}\times (S^1)^{n-1}$ diagonally. 
%\[(g_1,g_2)\cdot (x_1,\dots,x_{n-1},y_1,\dots,y_{n-1})=(g_1x_1,\dots,g_1x_{n-1},g_2y_1,\dots,g_2y_{n-1}).\]
First we prove the claim for $n=2 $ and $3$. Then, using these cases, we get an upper bound for the higher dimension cases.
\smallskip
\vspace{1mm}

\noindent\textbf{Case 1}: \emph{Suppose $n=2$.}
%\vspace{1mm}
In this case, $K_2$ is the Klein bottle and we have the fibre bundle $$S^1\xhookrightarrow{} K_2 \xrightarrow{\mathfrak{p}_2} \mathbb{R}P^1 ~\mbox{and} ~ S^1\times S^1\xhookrightarrow{}K_2\times K_2 \xrightarrow{\mathfrak{p}_2 \times \mathfrak{p}_2} \R P^1\times \R P^1.$$ 
%Following \Cref{cor:TC of gpps}, we have to find an $(\Z_{2}\times \Z_2)$-invariant open cover of fibre $S^1\times S^1$ such that each open set admits a local continuous section of fibration $\pi_{S^1}:PS^1\to S^1\times S^1$.
%%Let $A=\{e^{i\theta}\in S^1 \mid -\pi/2<\theta <\pi/2 \}$ and $B=\{e^{i\theta}\in S^1 \mid \pi/2<\theta <3\pi/2 \}$.
%Define, \[U_1= (S^1\times S^1)\setminus (1\times S^1\cup S^1\times 1) =(S^1\setminus \{1\})\times (S^1\setminus \{1\}),\]
%\[U_2=(A\times S^1\cup S^1\times A)\setminus A\times A=A\times B\sqcup B\times A,\] \[U_3=A\times A.\]
Let
$A :=\{e^{i\theta}\in S^1 \mid \theta \neq \pm\pi/2 \}$.
Define, \[U_1:= S^1\setminus \{1\}\times S^1\setminus \{1\}, ~~U_2:= S^1\setminus \{-1\}\times S^1\setminus \{-1\} ~\mbox{ and }~
U_3=A\times A.\]
Then $U_1, U_2$ and $U_3$ are $(\Z_2\times \Z_2)$-invariant and covers $S^1\times S^1$. Moreover, $U_3$ is a disjoint union of contractible subsets  and $U_1$, $U_2$ are contractible. Thus, the open cover  $\{U_1,U_2,U_3\}$ forms a categorical open cover of $S^1\times S^1$. So, $\{U_1, U_2, U_3\}$ is a motion planning cover for $S^1 \times S^1$.
 Therefore, \Cref{cor:TC of gpps} gives $\TC(K_2)\leq \mathrm{cat}(\R P^1\times \R P^1)+3-1=5$. Recall that \cite[Proposition 5.2]{DDavis} gives $5\leq \TC(K_2)$. Hence, $\TC(K_2)=5$. This has been shown in \cite{TCKleinbottle}. However, we wish to keep it for the readers convenience.
 \smallskip
\vspace{1mm}

\noindent\textbf{Case 2}: \emph{Suppose $n=3$.}
In this case, we have the fibre bundles $S^1\times S^1\xhookrightarrow{} K_3 \xrightarrow{\mathfrak{p}_3} \mathbb{R}P^1$ and $(S^1)^2\times (S^1)^2\xhookrightarrow{}K_3\times K_3 \xrightarrow{\mathfrak{p}_3 \times \mathfrak{p}_3} \R P^1\times \R P^1$. Let $$V_1:=\{(z_1,z_2)\in S^1\times S^1 \mid z_1\neq -z_2 \} \mbox{ and } V_2:=\{(z_1,z_2)\in S^1\times S^1 \mid z_1\neq z_2\}.$$ These are Farber's motion planning cover of $S^1$. Observe that both $V_1$ and $V_2$ are $\Z_2$-invariant open sets of $(S^1)^2$. For each $1\leq i, j\leq 2$, the set $V_i\times V_j$ is $(\Z_2\times \Z_2)$-invariant open set and $\{V_i \times V_j\}_{i,j=1}^2$ covers $(S^1)^2\times (S^1)^2$. Thus, the product of local continuous sections on $V_i$ and $V_j$ of $\pi_{S_1}$ gives a local continuous section on $V_i \times V_j$ of $\pi_{(S^1)^2}$ for $1\leq i,j\leq 2$. Therefore, $\TC(K_3)\leq \mathrm{cat}(\R P^1\times \R P^1)+4-1=6$. Now \cite[Proposition 5.2]{DDavis} gives $6\leq \TC(K_3)$. Hence, we have $\TC(K_3)=6$.
\smallskip
\vspace{1mm} 

\noindent\textbf{Case 3}: \emph{Suppose $n=2k+1$ where $k\geq 2$.}
In this case, we have the fibre bundles $(S^1)^{2k} \xhookrightarrow{} K_{2k+1} \xrightarrow{\mathfrak{p}_n} \mathbb{R}P^1$ and $(S^1)^{2k}\times (S^1)^{2k} \xhookrightarrow{} K_{2k+1}\times K_{2k+1} \xrightarrow{\mathfrak{p}_n \times \mathfrak{p}_n} \R P^1\times \R P^1$. 
%Note that we have to find $(\Z_2\times \Z_2)$-invariant open subsets of $(S^1)^{2k}\times (S^1)^{2k}$. 
We write 
\[\begin{split}(S^1)^{2k}\times (S^1)^{2k}&=(X_1\times \cdots \times X_k)\times (Y_1\times \cdots \times Y_k)\\
&= X_1\times Y_1\times\cdots \times X_k\times Y_k,
\end{split}
\]
where $X_i=(S^1)^2=Y_i$ for all $1\leq i\leq k$.
Using {\bf Case 2} and the proof of the (generalized) product formula for the topological complexity \cite[Theorem 4.2]{EqTCGrant}, we get a motion planning cover for $(S^1)^{2k}\times (S^1)^{2k}$ consisting of $4k-(k-1)=3k+1$ many $(\Z_2\times \Z_2)$-invariant open subsets. Therefore, \Cref{cor:TC of gpps} and \cite[Proposition 5.2]{DDavis} give $2k+4\leq\TC(K_{2k+1})\leq 3k+3$. 

\smallskip
\vspace{1mm}

\noindent\textbf{Case 4}: \emph{Suppose $n=2k$ where $k\geq 2$.} Then, we have the fibre bundles ${(S^1)^{n-1} \xhookrightarrow{} K_{n} \xrightarrow{\mathfrak{p}_n} \mathbb{R}P^1}$ and $(S^1)^{2k-1} \times (S^1)^{2k-1} \xhookrightarrow{} K_{2k}\times K_{2k} \xrightarrow{\mathfrak{p}_n \times \mathfrak{p}_n} \R P^1\times \R P^1$. 
We write 
\[\begin{split}(S^1)^{2k-1}\times (S^1)^{2k-1}&=((X\times Y)\times (X\times Y )\\
&= (X\times X)\times (Y\times Y),
\end{split}
\]
where $X=(S^1)^{2(k-1)}$ and $Y=S^1$. Using {\bf Case 3} and the proof of the product formula for the topological complexity  \eqref{eq_tc_prod}, we get a motion planning cover for ${(S^1)^{2k-1}\times (S^1)^{2k-1}}$ consisting of ${3k-2+3-1=3k}$ many $(\Z_2\times \Z_2)$-invariant open subsets. Therefore, \Cref{cor:TC of gpps} and \cite[Proposition 5.2]{DDavis} give $2k+3\leq\TC(K_{2k+1})\leq 3k+2$. 
\end{proof}

\section{LS-category and topological complexity of some generalized projective product spaces}\label{sec: Cat TC some gpps}
 Let $S^n :=\{(u_1, \ldots, u_{n+1})\in \R^{n+1} ~|~ u_1^2 + \cdots + u_{n+1}^2 = 1\}$. 
Davis \cite{Davis} introduced the projective product spaces and studied some of its topological properties. Let $(n_1,\dots,n_r)$ be a tuple of non-negative integers. Define
\begin{equation}\label{eq_ppsp}
P(n_1, \ldots, n_r):= \frac{S^{n_1}\times \cdots \times S^{n_r}}{({\bf x}_1, \dots, {\bf x}_{r})\sim (-{\bf x}_1, \dots, -{\bf x}_{r})}.
\end{equation}
Then $P(n_1, \ldots, n_r)$ is called a projective product spcace corresponding to the tuple $(n_1, \ldots, n_r)$.

%We recall the mod-$2$ cohomology algebra of $P(n_1, \ldots, n_r)$.

%\begin{theorem}[{\cite[Theorem 2.1]{Davis}}]
%Let $n_1\leq \dots\leq n_r$ If $n_1 < n_2$ or $n_1$ is odd, the mod-$2$ cohomology ring of $P(n_1,\dots,n_r)$ is given by \[H^{\ast}(P(n_1,\dots,n_r); \Z_2) = \Z_2[\alpha]/(\alpha^{n_1+1}) \otimes \Lambda[\alpha_2,\dots,\alpha_r]\] where $|\alpha| = n_1$, $|\alpha_i| = n_i$ for $i > 1$, and $\Lambda$ denotes an exterior algebra.  If is even and $n_1 = \dots = n_k < n_k+1$ for some $k > 1$, $H^{\ast}(P(n_1,\dots,n_r); \Z_2)$ is the same as above with the extra relation given by $\alpha_i^2 = \alpha^{n_1} \alpha_i$ for $2 \leq i \leq k$.    \end{theorem}
%We now define a class of generalized projective product spaces which contains projective product spaces.
%Let $M$ be a compact manifold with a free involution $\tau$. Then the following space \[X((n_1, p_1), \ldots, (n_r,p_r), M):= \frac{S^{n_1}\times \cdots \times S^{n_r} \times M}{(x_1, \dots, x_{r}, y)\sim (-x_1,\dots, -x_{r}, \tau{y})}\] is a manifold. 

Now, consider an involution $\tau_j$ on $S^{n_j} \subset \R^{n_j+1}$ defined as follows:
\begin{equation}\label{eq: invo prodsphere}
  \tau_j( (y_1, \ldots, y_{p_j}, y_{p_j+1}, \ldots, y_{n_j+1}) ):= (y_1, \ldots, y_{p_j}, -y_{p_j+1}, \ldots, -y_{n_j+1}),  
\end{equation}
for some $0 \leq p_j \leq n_j$ and $1\leq j\leq r$.
 Then we have $\Z_2$-action on the product $S^{n_1}\times \dots \times S^{n_{r}}$ via the product involution 
 \begin{equation}\label{eq: invo prod taui}
  \tau(r):=\tau_1\times \dots \times \tau_{r}   \end{equation}
    Note that if $p_j=0$ then $\tau_j$ acts antipodally on $S^{n_j}$ and if $p_j=n_j$, then $\tau_j$ is a reflection across the hyper-plane $y_{n_j+1}=0$ in $\R^{n_j+1}$.

Let $N$ be a topological space with a free involution $\sigma$. Consider the identification space: 
\begin{equation}\label{eq:prodsphere_N}
X((n_1, p_1), \ldots, (n_r,p_r), N) :=\frac{ S^{n_1}\times \cdots \times S^{n_r}\times N}{({\bf x}_1, \dots, {\bf x}_{r}, y)\sim (\tau_1({\bf x}_1),\dots ,\tau_r({\bf x}_r), \sigma(y))},
\end{equation}
where $\tau_j$ is a reflection defined as in \eqref{eq: invo prodsphere} for $1\leq j\leq r$. So, $X((n_1, p_1), \ldots, (n_r,p_r), N)$ is a generalized projective product space. 
In this section, we study the LS-category and topological complexity of these spaces for several $N$. We note that,
in particular, if $p_j=0$ for $j=1, \ldots, r$ and $N=S^{n_{r+1}}$ with the involution $\sigma$ given by the antipodal action, then $X((n_1, 0), \ldots, (n_r, 0), N)$ is the projective product space $P(n_1, \ldots, n_{r+1})$. % and \cite[Theorem 2.1]{Davis}

We consider the trivial sphere bundle $$ S^{n_1}\times \cdots \times S^{n_{i-1}} \times S^{n_i} \times N \to  S^{n_1}\times \cdots \times S^{n_{i-1}}\times N$$
for $i=2, \ldots, r$. This induces the sphere bundle 
\[ S^{n_i} \xhookrightarrow{}  X((n_1, p_1), \ldots, (n_i,p_i), N)\xrightarrow{q_i} X((n_1, p_1), \ldots, (n_{i-1},p_{i-1}), N)\] for $i=2, \ldots, r$. 
Similarly, we can have a sphere bundle 
\begin{equation}\label{eq_sp_bundle}
S^{n_1} \xhookrightarrow{}  X((n_1,p_1), N)\xrightarrow{q_1} N/\left<\sigma \right>.
\end{equation}

%Since $\tau$ is a free involution on $M$, $X(n_1, M)$ is a $S^{n_1}$-bundle on $\overline{M}=\frac{M}{y \sim \tau y}$. 
%More generally, we have $X((n_1,0) \ldots, (n_r,0) M)$ is a $S^{n_r}$-bundle over $X(n_1, \ldots, n_{r-1}, M)$ for $r \geq 2$. We may assume that $n_1 \leq \cdots \leq n_r$. 
Let $\alpha$ be the (first) Stiefel-Whitney class of the line bundle associated to the principal $\Z_2$-bundle $N \to N/\left<\sigma \right>$. We denote an exterior algebra by $\Lambda(-)$ and the total Steenrod square by ${\rm Sq} = \sum_{n \geq 0} {\rm Sq}^n$.

%\begin{theorem}\label{thm:cohom_pps_1}
%Let $n_1\leq \cdots \leq n_r$. Then  
%$H^*(X((n_1,0) \ldots, (n_r,0) M); \Z_2)$ is isomorphic as a graded $\Z_2$-algebra to  
%\begin{align*}
%H^*(\overline{M};\Z_2) \otimes\Lambda(\alpha_1, \ldots, \alpha_r),
%\end{align*} where $|\alpha_i|=n_i,$ for $1 \leq i \leq r $ and $Sq(\alpha_i)=(1+\alpha)^{n_i+1}\alpha_i$.
%\end{theorem}
%\begin{proof}
%We can follow the arguments in the proof of \cite[Theorem 2.1]{Davis}. 
%\end{proof}

\begin{theorem}\label{thm: cohoringYM}
Let $n_1 \leq \cdots \leq n_r$. Then $H^*(X((n_1, p_1), \ldots, (n_r,p_r), N);\Z_2)$ is isomorphic as a graded $\Z_2$-algebra to  
 \begin{align*}
H^*(N/\left<\sigma \right>; \Z_2) \otimes \Lambda(\beta_1,\ldots, \beta_r),
\end{align*} where  $|\beta_j|=n_j,$ ${\rm Sq}(\beta_j)=(1+\alpha)^{n_j+1-p_j} \beta_j$, and $p_j \geq 1$ for $1 \leq j \leq r$.
\end{theorem}
\begin{proof}
%We first assume that $p_j\geq 1$.
%We prove the case when $p_j=n_j$. The similar arguments will work when $p_j=0$.
Consider the trivial vector bundle $\R^{n_1+1} \times N \to N$. This induces an $(n_1+1)$-vector bundle $$\frac{\R^{n_1+1} \times N} {((x_1, \ldots, x_{n_1}, x_{n_1+1}, y) \sim (\tau_1(x_1, \ldots, x_{n_1}, x_{n_1+1}), \sigma(y)))} \xrightarrow{\eta_1} \frac{N}{(y \sim \sigma(y))} = N/\left<\sigma \right>.$$ 
Then we have the cofibre sequence $S(\eta_1) \to D(\eta_1) \to {\rm Th}(\eta_1)$ where $S(\eta_1)$ and $D(\eta_1)$ are the total spaces of the sphere bundle and the disk bundle of $\eta_1$, and ${\rm Th}(\eta_1)$ is the Thom space of $\eta_1$. So $S(\eta_1) \cong X((n_1,p_1), N)$ and $D(\eta_1) \simeq N/\left<\sigma \right>$. Thus we have the following long exact sequence
\begin{equation}\label{eq_split}
\cdots H^{\ast}({\rm Th} (\eta_1)) \to H^{\ast}(N/\left<\sigma \right>) \to H^{\ast}(X((n_1,p_1), N)) \to H^{\ast+ 1}({\rm Th}(\eta_1)) \to \cdots .
\end{equation}
The point $(1, 0, \ldots, 0) \in S^{n_1}$ is a fixed point under the involution $\tau_1$ on $S^{n_1}$. Thus the map $N \to S^{n_1} \times N$ defined by $y \mapsto ((1, 0, \ldots, 0), y)$ induces a section $s \colon N/\left<\sigma \right> \to X((n_1,p_1), N)$ of $\eta_1$ defined by $[y] \mapsto [((1, 0, \ldots, 0,0), y)]$. Therefore, \eqref{eq_split}
is a split exact sequence, that is;
$$H^{\ast}(X((n_1,p_1), N)) \cong H^{\ast}(N/\left<\sigma \right>) \oplus H^{\ast+ 1}({\rm Th}(\eta_1)).$$  Let  $\beta_1$ be the image of the Thom class (in $H^{n_1+1}({\rm Th} (\eta))$) under this isomorphism.
Then, by the Thom isomorphism, we have $$H^{\ast}(X((n_1,p_1), N)) \cong H^{\ast}(N/\left<\sigma \right>) \oplus H^{\ast}(N/\left<\sigma \right>)\cdot \beta_1.$$ Note that as a real bundle, $\eta_1$ is isomorphic to the sum of $p_1$ many trivial bundle $\epsilon$ and $(n_1+1-p_1)$ many the line bundle $\zeta_1$ associated with the double cover $N \to N/\left<\sigma \right>$. Therefore, the total Steenrod square and the total Stiefel-Whitney class have the following relation in our setting, using the arguments in \cite[Page 94]{MiSt}. 
\begin{align*}
{\rm Sq}(\beta_1)&=W(n_1\epsilon \oplus (n_1+1-p_1)\zeta_1) \beta_1 \\ &= (1+\alpha)^{n_1+1-p_1}\beta_1
\end{align*} 
where $\alpha = w_1(\zeta_1)$ and $|\beta_1| = n_1$. So, $\beta^2_1 = \binom{n_1+1-p_1}{n_1} \alpha^{n_1} \beta_1$ if $p_1=1$ and $\beta_1^2 = 0$ if $p_1 > 1$. 

One can prove the claim for $r>1$ inductively using similar arguments and the naturality of the Steifel-Whitney classes. 
\end{proof}

\begin{remark} We observe the following things.
\begin{enumerate}
    \item When $r=1$ and $n_r=1$, then \Cref{thm: cohoringYM} is \cite[Lemma 2.1]{DDavis}.

    \item Suppose $m_1\leq \dots\leq m_k \leq n_1 \leq \cdots \leq n_r$ and  $N=S^{m_1}\times \dots \times S^{m_k}$ with the involution $\sigma$ generated by the antipodal on each sphere $S^{m_j}$ for $1\leq j\leq k$. Then $N/\left<\sigma \right>$ is a projective product space $P(m_1,\dots,m_k)$, and  \Cref{thm: cohoringYM} coincides with \cite[Theorem 4.1]{sarkargpps}. 
\end{enumerate}

\end{remark}

Let $R$ be a commutative ring with unity and $X$ a path connected topological space with its cohomology ring $H^{\ast}(X;R)$. The cup-length of $X$ over $R$ is the maximal number $r$ such that there exists $x_i\in H^{\ast}(X;R)$ for $i=1, \ldots, r$ satisfying $\prod_{i=1}^{r}x_i\neq 0$. We denote this number by ${\rm cl}_{R}(X)$. It is well known \cite[Proposition 1.5]{CLOT} that the number gives a lower bound for the $\ct(X)$.
Let \[\cup \colon H^{\ast}(X;R)\otimes H^{\ast}(X;R) \longrightarrow H^{\ast}(X;R)\] be the cup product. Then the zero-divisor cup-length of $X$ with respect to coefficient ring $R$ is defined as the maximal number $k$ such that there exist elements $u_i\in H^{\ast}(X;R)\otimes H^{\ast}(X;R)$ satisfying $\cup(u_i)=0$ for all $1\leq i \leq k$ and $\prod_{i=1}^{k} u_i\neq 0$. We denote this number by ${\rm zl}_{R}(X)$. It is known \cite[Theorem 7]{FarberTC} that the number ${\rm zl}_{R}(X)$ gives a lower bound for the $\TC(X)$. That is, we have 
\begin{equation}\label{eq:lscat_tc_lbd}
{\rm cl}_{R}(X)+1 \leq \ct(X) ~~ \mbox{and} ~~ {\rm zl}_{R}(X) +1\leq \TC(X).
\end{equation}

%{\color{red} \bf{Should we use notations $\ell_{c,R}(X)$ and $\ell_{z,R}(X)$ everywhere?}} ${\rm cl}_R(X)$ and  ${\rm zl}_R(X)$

We now compute some bounds on the LS-category and the topological complexity of $X((n_1,p_1) \ldots, (n_r,p_r), N)$.
%when $p_j=0$ for all $1\leq j\leq r$. 
\begin{proposition}\label{Prop_cat}
%Let $\ell_c$ be the cup-length of $H^*(\overline{M}; \Z_2)$. Then
Let $1 \leq p_j \leq n_j$. Then,
we have the following inequalities. 
\[{\rm cl}_{\Z_2}(N/\left<\sigma \right>)+r+1\leq \mathrm{cat}(X((n_1,p_1) \ldots, (n_r,p_r), N))\leq \mathrm{cat}(N/\left<\sigma \right>)+r.\]   
\end{proposition}
\begin{proof}
%First we determine the lower bound on the LS-category using the cup-length of $H^*(\overline{M}; \Z_2)$.
Let $u=\prod_{i=1}^{{\rm cl}_{\Z_2}(N/\left<\sigma \right>)} y_i$ be a largest non-zero product in $H^*(N/\left<\sigma \right>; \Z_2)$. Then the product $u\cdot\prod_{j=1}^{r}\beta_j$ is non-zero in $H^{\ast}(X((n_1,p_1) \ldots, (n_r,p_r), N);\Z_2)$. Then we get the left inequality using \eqref{eq:lscat_tc_lbd}.

We use  \Cref{prop cat} to obtain the upper bound.
Let $e_{n_i}(1)=(1, 0, \ldots, 0) \in S^{n_i}$.
%Note that $\{U_{1_i}:=S^{n_i}\setminus \{\pm e_{n_i}(i)\}, U_{2_i}:=S^{n_i}\setminus S^{n_i-1}\times 0\}$ forms a categorical cover for the sphere $S^{n_i}$ for $i=1, \ldots, r$.
Consider $U_{i_1}=S^{n_i}\setminus \{e_{n_i}(1)\}$ and $U_{i_2}=S^{n_i}\setminus \{-e_{n_i}(1)\}$. Then, the set $\{U_{i_1}, U_{i_2}\}$ is an {$\left<\tau_i\right>$-invariant} open cover of $S^{n_i}$. Observe that they are also $\left<\tau_i \right>$-categorical for $i=1, \ldots, r$ and the fixed point set of the $\left< \tau_i \right>$-action on $S^{n_i}$ is path connected.
%I think we only need this much. But for the \Cref{tc gp} we need to have a categorical cover.( The question is, whether this cover is $\left<\tau_i \right>$-categorical. For $p_j=1$ this is clear. ) 
%Observe that this categorical cover is invariant under the antipodal action. 
Therefore, using \cite[Theorem 2.23]{eqicatprodineq}, we have $\ct_{\Z_2}(\prod_{i=1}^{r}S^{n_i}) \leq r+1$. Thus, one can construct a categorical open cover of $\prod_{i=1}^{r}S^{n_i}$ with $r+1$ many open sets invariant under the $\Z_2$-action. Thus, we get the right inequality using \Cref{prop cat}.
\end{proof}

\begin{remark}
We observe the following things.
\begin{enumerate}
\item Recall that the LS-category of the projective product spaces was computed in \cite{Vandembroucq}. If $N=S^{n_{r+1}}$ and $\sigma$ is the antipodal action on $N$, then \Cref{Prop_cat} coincides with \cite[Theorem 1.2]{Vandembroucq}.

\item Suppose $m_1\leq \dots\leq m_k$ and  $N=S^{m_1}\times \dots \times S^{m_k}$ with the involution $\sigma$ generated by the antipodal on each sphere $S^{m_j}$ for $1\leq j\leq k$. Then $N/\left<\sigma \right>=P(m_1,\dots,m_k)$. Recall that the cup-length of $P(m_1,\dots,m_k)$ is $m_1+k-1$ and $\ct(P(m_1,\dots,m_k))=m_1+k$. Therefore, $\ct((X((n_1,p_1) \ldots, (n_r,p_r), N)))=m_1+k+r$ if $1 \leq p_i \leq n_i$ for $i=1, \ldots, r$.
\end{enumerate}
\end{remark}

%Recall the involution $\tau_i$ on the sphere $S^{n_i}$ defined in \eqref{eq: invo prodsphere}. Suppose that $\Z_2 \cong \left < \tau_i \right >$ acts on $S^{i}$ via $\tau_i$. We compute the $\Z_2$-equivariant topological coplexity of $S^{n_i}$. 

%{\color{red} Let $\tau_i: S^{n_i}\to S^{n_i}$ defined by \[\tau_i(u_1,\dots,u_{p_i},u_{p_i+1},\dots, u_{n_i+1})=(u_1,\dots,u_{p_i},-u_{p_i+1},\dots, -u_{n_i+1}).\]  Let $\Z_2=<\tau_i>$. 
  
  %If we know that $\ct_{\Z_2}(S^{n_i})=2$, then we have $\TC_{\Z_2}(S^{n_i})=3$ if $n_i$ is even. %If we know that $\ct_{\Z_2}(S^{n_i})=2$, then we get $\TC_{\Z_2}(S^{n_i})=3$ if $p_i$ is odd and $2\leq \TC_{\Z_2}(S^{n_i})\leq 3 $ if $p_i$ is even. }

We now prove the corresponding result for the topological complexity.
\begin{proposition}\label{prop:TCrefaction}
Let $2\leq n_1 \leq \cdots \leq n_r$  and $p_j >1$ for $j=1, \ldots, r$.
Then  
\begin{equation}\label{eq: TCYM}
   {\rm zl}_{\Z_2}(N/\left<\sigma \right>)+ r+1\leq\TC(X((n_1,p_1)\dots,(n_r,p_r),N)\leq \ct(N/\left<\sigma \right>\times  N/\left<\sigma \right>)+ 2r. 
\end{equation}
\end{proposition}
\begin{proof}

The following inequality follows from \Cref{cor:TC of gpps} 
\[\TC(X((n_1,p_1)\dots,(n_r,p_r),N)\leq q+ \ct(N/\left<\sigma \right>\times N/\left<\sigma \right>)-1,\]
where $q$ is the cardinality  of a $(\left<\tau(r)\right>\times \left<\tau(r)\right>)$-invariant motion planning cover of $\prod_{i=1}^rS^{n_i}$.
%One can see that $q\leq \TC^*_{\tau^r}(\prod_{i=1}^rS^{n_i})$.
%It follows from \cite[Proposition 4.25]{surveyeqtc} that \[\TC^*_{\left<\tau^r\right>}(\prod_{i=1}^rS^{n_i}) \leq \ct_{\left<\tau^r\times \tau^r\right>}(\prod_{i=1}^rS^{n_i}\times \prod_{i=1}^rS^{n_i}).\]
Since the space $\prod_{i=1}^rS^{n_i}$ is a $\left<\tau(r)\right>$-connected space,  we get the following inequality using \cite[Proposition 2.10]{BaySarkarheqtc}.
\[\ct_{\left<\tau(r)\times \tau(r)\right>}(\prod_{i=1}^rS^{n_i}\times \prod_{i=1}^rS^{n_i})\leq 2\ct_{\left< \tau(r)\right>}(\prod_{i=1}^r S^{n_i})-1.\]
Now it follows from product inequality of equivariant LS-category \cite[Proposition 2.9]{BaySarkarheqtc} that 
\[\ct_{\left< \tau(r)\right>}(\prod_{i=1}^r S^{n_i})\leq \sum_{i=1}^r\ct_{\left <\tau_i\right>}(S^{n_i})-(r-1)=r+1,\] since $\ct_{\left <\tau_i\right>}(S^{n_i})=2$ as shown in the proof of \Cref{Prop_cat}.
Therefore, one gets $(2r+1)$ many $(\left<\tau(r)\right>\times \left<\tau(r)\right>)$-invariant motion planning cover of $\prod_{i=1}^rS^{n_i}$. Thus, we get the right inequality of \eqref{eq: TCYM}.
%Since for each $i$, $n_i\geq 2$ and $p_j\geq 2$, the fixed point set  $(\prod_{i=1}^{r}S^{n_i})^{\Z_2}$ is path connected. Therefore, by \cite[Corollary 5.8]{EqTCGrant}, we have $\TC_{\Z_2}(\prod_{i=1}^{r}S^{n_i})\leq 2\ct_{\Z_2}(\prod_{i=1}^{r}S^{n_i})-1$. Using product inequality of equivariant LS-category (see \cite[Theorem 2.23]{eqicatprodineq}), it can be shown that $\ct_{\Z_2}(\prod_{i=1}^{r}S^{n_i}) \leq r+1$. Thus we get the right inequality of \eqref{eq: TCYM}.

The left inequality of \eqref{eq: TCYM} follows from the zero-divisor cup-length calculation using \Cref{thm: cohoringYM}. This proves the proposition. 
\end{proof}

Let $\Sigma_g$ be the orientable surface of genus $g$ embedded in $\R^3$ such that it admits the antipodal action. Observe that the quotient of $\Sigma_g$ by the antipodal action is the non-orientable surface $N_{g+1}$ of genus $g+1$. 
 Then the following results are straightforward consequences of \Cref{Prop_cat} and \Cref{prop:TCrefaction}.
\begin{corollary}\label{cor:tc_sigmag}
    Let $2\leq n_1 \leq \cdots \leq n_r$  and $p_j >1$ for $j=1, \ldots, r$. Then 
\begin{enumerate}
    \item $\ct(X((n_1,p_1)\dots,(n_r,p_r), \Sigma_g))=r+3$.
    \item $r+4\leq \TC(X((n_1,p_1)\dots,(n_r,p_r), \Sigma_g))\leq 2r+5$.
\end{enumerate}
\end{corollary}

Next, we consider the following class of generalized projective product spaces.

\begin{equation}\label{eq: Xgn}
  X_g^{n-2}:= \frac{S^1\times \cdots \times S^1 \times \Sigma_g}{(z_1,\dots ,z_{n-2},x)\sim (\bar{z}_1,\dots ,\bar{z}_{n-2},-x) }.  
\end{equation}
Note that, for any $g\geq 0$ and $n \geq 2$, the space $X_g^{n-2}$ is a closed smooth manifold of dimension $n$. If $n=2$, $X_g^{n-2}=N_{g+1}$.  Moreover, we have a fibre bundle $(S^1)^{n-2}\xhookrightarrow{} X_g^{n-2}\to N_{g+1}$. Observe that $X_g^{n-2}=X((n_1,p_1),\dots,(n_r,p_r), \Sigma_g)$, where $r=n-2$ and $n_i=p_i=1$ for $1\leq i\leq n-2$.
The following proposition is a consequence of \cite[Lemma 2.1]{DDavis} and \Cref{thm: cohoringYM}.
\begin{proposition}
The cohomology ring of $X_g^{n-2}$ is given by the following.
\[H^{\ast}(X_g^{n-2};\Z_2)=\displaystyle\frac{\Z_2\left< x_1,\dots x_{g+1},y_1,\dots,y_{n-2} \right>}{\left< \{x_ix_j,~ x_i^3,~ y_s^2-w_1y_s \mid 1\leq i\neq j\leq g+1,~ 1\leq s\leq n-2\} \right>},\]
where $|x_i|=1=|y_i|$ and $w_1\in H^1(N_{g+1};\Z_2)$ classifies the double cover $\Sigma_g \to N_{g+1}$.
\end{proposition}

Now we are ready to compute the LS-category and some bounds for the topological complexity of $X_g^{n-2}$.
\begin{proposition}
 The LS-category of $X_g^{n-2}$ is $\mathrm{cat}(X_g^{n-2})=n+1$.   
\end{proposition}
\begin{proof}
From the description of the cohomology ring of $X_g^{n-2}$, one can show that the product $x_1\cdot w_1\cdot\prod_{i=1}^{n-2}y_i$ is non-zero in $H^{\ast}(X_g^{n-2};\Z_2)$. Therefore, $n+1\leq \mathrm{cat}(X_g^{n-2})$.  We also have $\ct (X_g^{n-2}) \leq \mathrm{dim} (X_g^{n-2})+1 = n+1$ from \cite[Theorem 1.7]{CLOT}. 
\end{proof}

\begin{theorem} 
Let $n\geq 3$. Then we have the following bounds for $\TC(X_g^{n-2})$. 
\[n+4\leq \TC(X_g^{n-2})\leq 
\begin{cases}
3k+2 & \text{ if } $n=2k$,\\
3k+4& \text{ if } $n=2k+1$.
\end{cases}\]    
In particular, $\TC(X_g^1)=7$ and $\TC(X_g^2)=8$.
\end{theorem}
\begin{proof}
First we compute a lower bound for the topological complexity using zero-divisor cup-length.
Consider the basic divisors $X_1=1\otimes x_1 + x_1\otimes 1$ and $Y_i=1\otimes y_i+y_i\otimes 1$ for $ 1 \leq i\leq n-2$. Consider the product 
\begin{align*}
X_1^3 Y_1^3 \prod_{i=2}^{n-2}Y_i & = (1\otimes x_1 + x_1\otimes 1)^3  (1\otimes y_1 + y_1\otimes 1)^3 \prod_{i=2}^{n-2} ( 1\otimes y_i + y_i\otimes 1)\\
& = (x_1^2\otimes x_1+x_1\otimes x_1^2) (w_1^2y_1\otimes 1 +w_1y_1\otimes y_1+y_1\otimes w_1y_1+1\otimes w_1^2y_1)  \prod_{i=2}^{n-2}Y_i,
 \end{align*}
%\[X_1^3\cdot Y_1^3\cdot\prod_{i=2}^{n-2}Y_i =(x_1^2\otimes x_1+x_1\otimes x_1^2)\cdot (w_1^2y_1\otimes 1 +w_1y_1\otimes y_1+y_1\otimes w_1y_1+1\otimes w_1^2y_1)\cdot\prod_{i=2}^{n-2}Y_i. \] 
Note that the above product contains a term $x_1w_1 \prod_{i=1}^{n-2}y_i\otimes x_1^2y_1$ which is not killed by any other term in the product. Therefore, $n+4\leq \TC(X_g^{n-2})$. 

 Consider the fibre bundle \[(S^1)^{n-2}\times (S^1)^{n-2}\xhookrightarrow{} X_g^{n-2}\times X_g^{n-2}\to N_{g+1}\times N_{g+1}.\]
We have to find a categorical cover of $N_{g+1}\times N_{g+1}$ and $(\Z_2\times \Z_2)$-invariant motion planning cover of $(S^{1})^{n-2}\times (S^{1})^{n-2}$. Recall that we have computed such an open cover in \Cref{TCKn} consisting 3k many open sets if $n=2k+1$ and $3k-2$ many open sets if $n=2k$. Using Kunneth formula, we get the cup-length of $H^*(N_{g+1}\times N_{g+1};\Z_2)=4$, since the cup-length of $H^*(N_{g+1};\Z_2)=2$. Therefore, $\mathrm{cat}(N_{g+1}\times N_{g+1})=5$. Thus, we get the desired upper bound using \Cref{cor:TC of gpps}.
\end{proof}

%\section{Equivariant LS-category and topological complexity}\label{sec: eq Cat TC some gpps}
%We end this section with the following observation.
%\begin{remark} Observe that the $n$-dimensional Klein bottle $K_n=X((n_1,p_1),\dots,(n_r,p_r),N)$, where for each $1\leq i \leq r=n-1$, $n_i=1$ and $N=S^1$. We also have defined spaces $X_g^{n-2}$ in \eqref{eq: Xgn}. We can see that $X_g^{n-2}=X((n_1,p_1),\dots,(n_r,p_r),N)$, where for each $1\leq i \leq r=n-2$, $n_i=1$ and $N=\Sigma_g$, the orientable surface of genus $g$. Note that  \Cref{prop:TCrefaction} does not give an upper bound on the topological complexity of $K_n$ and $X_g^{n-2}$, since in these cases the fixed point set of $\Z_2$ action  $(\prod_{i=1}^{n-1}S^{1})^{\Z_2}$ is not path connected. \end{remark}

%{\color{red} \bf{Should we shift things from Theorem 4.6 to Remark 4.10 to the begining of Section 5?}}

Let $\tau$ be an involution on $M$. This map  induces an automorphism $$\tau^* \colon H^*(M; \Z_2) \to H^*(M; \Z_2).$$ 
\begin{theorem}\label{thm:toric_mod2}
Let $M$ be a compact, simply connected and path-connected space with an involution $\tau$ such that $\tau^*$ is identity. Let $N$ be a simply connected path-connected space with a free involution $\sigma$. Then
$$H^*(X(M, N)); \Z_2) \cong H^*(N/\left<\sigma \right>; \Z_2) \otimes H^*(M; \Z_2).$$
\end{theorem} 
\begin{proof}
The projection $M \times N \to N$ induces a fibre bundle  $M \xhookrightarrow{} X(M, N) \xrightarrow{\mathfrak{p}} N/\left<\sigma \right>$ with fibre $M$, see \eqref{eq:indc_fib_bndl}. By the hypothesis, we have $\pi_1(X(M, N) = \Z_2$. Since $M$ is compact, the cohomology groups of the fibre and base of this bundle have finite dimension over the field $\Z_2$. The fibre $M$ and the base $N/\left<\sigma \right>$ are path connected, and the map $\tau^*$ is identity. Then, by applying \cite[Proposition 5.5]{Mcc}, one gets that the corresponding spectral sequence collapses at $E_2$. Hence $M$ is totally non-homologous to zero in $X(M, N)$ with respect to $\Z_2$. Therefore, by \cite[Theorem 5.10]{Mcc}, one gets the claim in the theorem.
\end{proof}

If $n_1, \ldots, n_r$ are positive integers greater than one then $N=\prod_{i=1}^r S^{n_i}$ is simply connected. Consider the involution $\sigma$ on $N$ determined by the antipodal action on each factors. Let us denote the generalized projective product space for this $M$ and $N$ by $X(M, n_1, \ldots, n_r)$. Note that if $M=N$ with a free involution and $p_j=0$ for $j=1, \ldots, r$ in \eqref{eq:prodsphere_N} then $X((n_1, 0), \ldots, (n_r, 0), N) = X(M, n_1, \ldots, n_r)$. We have the following result. 

\begin{corollary}\label{cor:toric_mod2}
 Let $M$ be a compact, simply connected and path-connected space with an involution $\tau$ such that $\tau^*$ is identity, and $n_1, \ldots, n_r$ be positive integers greater than one. Then
$H^*(X(M, n_1, \ldots, n_r)); \Z_2) \cong H^*(P(n_1, \ldots, n_r);\Z_2) \otimes H^*(M; \Z_2).$   
\end{corollary}

\begin{proposition}
Let $M$ be a compact, simply connected and path-connected space with an involution $\tau$ such that $\tau^*$ is identity, and ${\rm cl}_{\Z_2}(M)$ be the cup-length of $H^*(M; \Z_2)$ and $2 \leq n_1 \leq \cdots \leq n_r$. Then $ {\rm cl}_{\Z_2}(M) + r + n_1 \leq \mathrm{cat}(X(M, n_1, \ldots, n_r)) $.   
\end{proposition}
\begin{proof}
Let $\prod_{i=1}^{{\rm cl}_{\Z_2}(M)} y_i$ be non-zero in $H^*(M; \Z_2)$. Then the cup product $\prod_{i=1}^{{\rm cl}_{\Z_2}} y_i \alpha_1^{n_1} \prod_{j=2}^{r}\alpha_i$ is non-zero in  $H^{\ast}(X(M, n_1, \ldots, n_r); \Z_2)$, by Corollary \ref{cor:toric_mod2}. Therefore, the result follows from \eqref{eq:lscat_tc_lbd}.
\end{proof}

\begin{proposition}\label{prop:tc-bundle}
Let $M$ be a compact, simply connected and path-connected space with an involution $\tau$ such that $\tau^*$ is identity. Let ${\rm zl}_{\Z_2}(M)$ be the zero-divisor cup-length of $H^*(M; \Z_2)$ and $2 \leq n_1 \leq  \cdots \leq n_r$. Then $ {\rm zl}_{\Z_2}(M)+ {\rm zl}_{\Z_2}(\R P^{n_1}) + r \leq \mathrm{TC}(X(M, n_1, \ldots, n_r)) $.   
\end{proposition}
\begin{proof}
This follows from \Cref{thm:toric_mod2} and the computation of the zero-divisor cup-length of the space $X(M, n_1, \ldots, n_r)$.
\end{proof}
%\begin{remark}
%If the space $M$ in \Cref{prop:tc-bundle} satisfies the hypothesis in \Cref{thm:TC-fibrebundle}, then $\mathrm{TC}(X(M, n_1)) \leq \TC(\mathbb{R} P^{n_1}) +  \TC(M)-1$. In particular, if $M$ is the product of a finitely many spheres with anti-podal actions, then we get \cite[Theorem 1.3]{Vandembroucq}. 
%\end{remark}
\section{LS-category and topological complexity of some generalized Dold manifolds}\label{sec: Cat TC some gen Dold}
In this section, we study LS-category and topological complexity of certain generalized projective product spaces called `Dold manifolds of Grassmann type'. 
Let $1\leq d < n$ and $\Gr_d(\C^n)$ be the set of all $d$-dimensional subspaces of $\C^n$. The space $\Gr_d(\C^n)$  is called a complex Grassmann manifold and its complex dimension is $d(n-d)$. We note that $\Gr_1(\mathbb{C}^n) = \mathbb{C} P^{n-1}$. Let \[X(\Gr_d(\C^n), n_1,\dots,n_r):=\frac{\Gr_d(\C^n) \times S^{n_1}\times \cdots \times S^{n_r}}{(y, {\bf x}_1, \dots, {\bf x}_{r})\sim (\tau(y), -{\bf x}_1, \dots, -{\bf x}_{r})},\] where $\tau$ is the complex conjugation involution on $\Gr_d(\C^n)$, whose fixed point set is the real Grassmann manifold $\Gr_d(\mathbb{R}^n)$. We call $ X(\Gr_d(\C^n), n_1,\dots,n_r)$ a \emph{Dold manifold of Grassmann type}. We have a fibre bundle 
$$\Gr_d(\C^n)\xhookrightarrow{}X(\Gr_d(\C^n), n_1,\dots,n_r)\xrightarrow{\mathfrak{p}}P(n_1,\dots,n_r),$$ where $P(n_1,\dots,n_r)$ is the projective product space defined in \eqref{eq_ppsp}. If $n_1 \leq \cdots \leq n_r$, by Theorem \ref{thm:toric_mod2}, we have 
\begin{equation}\label{eq:cohm_dm}
H^*(X(\Gr_d(\C^n), n_1, \ldots, n_r)); \Z_2) \approx H^*(P(n_1, \ldots, n_r);\Z_2) \otimes H^*(\Gr_d(\C^n); \Z_2).
\end{equation}
We recall the cell structure on $\Gr_d(\C^n)$. A $d$-tuple $\lambda = (\lambda_1, \ldots, \lambda_d)$ is called a Schubert symbol if $1 \leq \lambda_1 < \cdots < \lambda_d \leq n$. Consider $\C ^{\ell} := \{(z_1, \ldots, z_{\ell}, 0, \ldots, 0) \in \C^n\}$. For a Schubert symbol $\lambda =(\lambda_1, \ldots, \lambda_d)$, define $$E(\lambda):=\{H \in \Gr_d(\C^n) ~|~ \dim(H \cap \C^{\lambda_j})=j, \dim(H \cap \C^{\lambda_j-1})=j-1 ~\mbox{for}~j=1, \ldots, d\}.$$  Then $E(\lambda)$ is even dimensional and is called a Schubert cell for the Schubert symbol $\lambda$. These  Schubert cells give a cell structure on $\Gr_d(\C^n)$, see \cite{MiSt}. The cup-length of  $\Gr_d(\C^n)$ is $d(n-d)$. Observe that each Schubert cell is invariant under the conjugation action and any two same dimensional cells corresponding to different Schubert symbols are disjoint. 

It is well known that the integral cohomology ring of $\Gr_d(\C^n)$ is described as follows: \[H^*(\Gr_d(\C^n),\Z)=\frac{\Z[c_1,\dots,c_d]}{\left<h_{n-d-1},\dots,h_n\right>},\]
where $|c_i|=2i$ and $h_j$ is defined as the $2j$-th degree term in the series expansion of $(1+c_1+\dots+c_d)^{-1}$.  Note that $\Gr_d(\C^n)$ is a K\"{a}hler manifold with ${\rm rank}(H^2(\Gr_d(\C^n);\Z)=1$. Therefore, $c_1^{d(n-d)}\neq 0$. Thus ${\rm cl}_{\Z}(\Gr_d(\C^n))=d(n-d)$. Since, $\Gr_d(\C^n)$ is simply connected and $\dim \Gr_d(\C^n)=2d(n-d)$, we have $\ct(\Gr_d(\C^n))\leq d(n-d)+1$. Therefore, we get $\ct(\Gr_d(\C^n))=d(n-d)+1$.

\begin{theorem}\label{thm_lscat_gras}
Let $n_1 \leq \cdots \leq n_r$. Then ${\rm cat}(X(\Gr_d(\C^n), n_1, \ldots, n_r)) = d(n-d) + n_1+r$.
\end{theorem}
\begin{proof}
The cohomology of the projective product spaces and \eqref{eq:cohm_dm} give that the cup-length of $X(\Gr_d(\C^n),n_1, \ldots, n_r)$ is $d(n-d) + n_1+r-1$. So, \[{\rm cat}(X(\Gr_d(\C^n),n_1, \ldots, n_r)); \Z_2) \geq d(n-d) + n_1+r.\] 
Let $0 \leq i \leq d(n-d)$ and $U_i:= \bigcup_{|\lambda|\leq i} E(\lambda)$. So, $U_{i}$ is a subcomplex of $\Gr_d(\C^n))$. Thus there is a conjugation invariant open neighborhood $V_{i}$ of $U_{i}$ such that $V_{i}$ retracts on $U_{i}$. Then, $\{V_{i} - V_{i -1}\}_{i=0}^{d(n-d)}$ gives a conjugation invariant categorical cover with $d(n-d)+1$ many open sets. Here $V_{-1}=\emptyset$. Therefore, by \Cref{prop cat} and ${\rm cat}({P(n_1,\dots,n_r)})=n_1 +r$, we get the following inequality \[{\rm cat}(X(\Gr_d(\C^n), n_1, \ldots, n_r)) \leq d(n-d) + n_1+r.\]
This proves the claim. 
\end{proof}

\begin{proposition}\label{prop_tc_grdn}
Let $n_1 \leq \cdots \leq n_r$. Then 
\begin{equation}\label{eq:tc-GRpps}
{\rm zl}_{\Z_2}(\Gr_d(\C^n))+ {\rm zl}_{\Z_2}(\R P^{n_1}) +r \leq\TC(X(\Gr_d(\C^n), n_1, \ldots, n_r)) \leq 2d(n-d) + 2(n_1+r)-1.  
\end{equation}
\end{proposition}
\begin{proof}
Using \eqref{eq:cohm_dm}, one can get that 
 \[{\rm zl}_{\Z_2}((X(\Gr_d(\C^n), n_1, \ldots, n_r))={\rm zl}_{\Z_2}(\Gr_d(\C^n))+ {\rm zl}_{\Z_2}(\R P^{n_1}) +r-1.\]
 So, the left inequality of \eqref{eq:tc-GRpps} follows.  
 
To get the right inequality, we need to find a $(\left<\tau \right> \times \left< \tau \right>)$-invariant motion planning cover of $\Gr_d(\C^n)\times \Gr_d(\C^n)$.
Note that the fixed point set of the $\left< \tau \right>$-action on $\Gr_d(\C^n)$ is a path-connected set, and $\Gr_d(\C^n)$ can be covered by $d(n-d)+1$ many conjugation invariant categorical sets. 
Therefore, using the product inequality for the equivariant {LS-category} in \cite[Proposition 2.9]{BaySarkarheqtc}, we can find a $(\Z_2\times \Z_2)$-invariant motion planning cover of $\Gr_d(\C^n)\times \Gr_d(\C^n)$ with $2d(n-d)+1$ many open sets. Since ${\rm cat}({P(n_1,\dots,n_r)})=n_1 +r$, we get that ${\rm cat}(P(n_1,\dots,n_r)\times P(n_1,\dots,n_r))\leq 2(n_1+r)-1$. These observations  prove the right inequality of \eqref{eq:tc-GRpps}.
\end{proof}

\begin{theorem}\label{thm:cat_CPnN}
Let $N$ be a simply connected space with a free involution $\sigma$ and $\mathbb{C}P^{n}$ with the complex conjugation involution. Then
\begin{equation}\label{eq:cl-cat}
n + {\rm cl}_{\Z_2}(N/\left<\sigma \right>)+1 \leq \mathrm{cat}(X(\mathbb{C} P^n, N))\leq n+ \rm{cat}(N/\left<\sigma \right>).
\end{equation}
%In particular, $\mathrm{cat}({X}(\mathbb{C} P^n, n_1,\dots,n_r)) = n_1+ r+n.$
\end{theorem}
\begin{proof}
Using \Cref{thm:toric_mod2}, we get ${\rm cl}_{\Z_2}{X}(\mathbb{C} P^n, N))=n + {\rm cl}_{\Z_2}(N/\left<\sigma \right>)$. Therefore, we have $\mathrm{cat}({X}(\mathbb{C} P^n, N)) \geq  n +{\rm cl}_{\Z_2}(N/\left<\sigma \right>) +1.$ 
Note that $\mathrm{cat}_{\Z_2}(\C P^n)=n+1$, see the proof of \cite[Proposition 3.10]{Naskar}.
Then, using \Cref{thm:cat}, we get the right inequality of \eqref{eq:cl-cat}.
\end{proof}
\begin{remark}
Observe that if $\mathrm{cat}(N/\left<\sigma \right>)={\rm cl}_{\Z_2}({N/\left<\sigma \right>})+1$ then   $\mathrm{cat}(X(\mathbb{C} P^n, N))= n+ {\rm cat}(N/\left<\sigma \right>)$. In particular, $\mathrm{cat}(X(\mathbb{C} P^n, n_1,\dots,n_r))= n+ n_1+r$, and $\mathrm{cat}(X(\mathbb{C} P^n, n_1))= n+ n_1+1$ which is \cite[Corollary 2.7]{Naskar}.
\end{remark}

\begin{theorem}\label{thm:TC-cpn-pn}
Let $n_1 \leq \cdots \leq n_r$ and $M=\mathbb{C}P^{n}$ with the conjugation involution. Then
\begin{equation}\label{eq:tc_ineq}
{\rm zl}_{\Z_2}(\mathbb{C}P^{n}) + r + {\rm zl}_{\Z_2}(\mathbb{R}P^{n_1}) \leq \TC(X(\mathbb{C} P^n, n_1,\dots,n_r)) \leq 2(n_1+r + n)-1,
\end{equation} 
where $k$ is the number of evens in ${n_1, \ldots, n_r}$. 
\end{theorem}
\begin{proof}
The left inequality of \eqref{eq:tc_ineq} can be obtained from \Cref{thm:toric_mod2} and the computation of the zero-divisor cup-length of $X(\mathbb{C} P^n, n_1,\dots,n_r)$.

Recall that the categorical cover of $\C P^n$ consists of $(n+1)$ many open sets homeomorphic to $\C^n$. These open sets are invariant under the $\left< \tau \right>$-action on $\C P^n$,  where $\tau$ is the conjugation involution on $\C P^n$. Therefore, using \cite[Proposition 2.9]{BaySarkarheqtc}, we can have $(\left<\tau \right> \times \left< \tau \right>)$-invariant categorical cover of $\C P^n\times \C P^n$ consists of $(2n+1)$ many open sets. Thus, this categorical cover of $\C P^n\times \C P^n$ can be regarded as a$(\left<\tau \right> \times \left< \tau \right>)$-invariant motion planning cover, since $\TC(\C P^n)=2n+1$. Thus, using \Cref{cor:TC of gpps}, we get the following inequality. \[\TC(X(\mathbb{C} P^n, n_1,\dots,n_r))\leq \ct(P(n_1,\dots,n_r)\times P(n_1,\dots,n_r))+2n+1-1.\]
Therefore, the right inequality of \eqref{eq:tc_ineq} follows from \cite[Theorem 1.2]{Vandembroucq} and \cite[Theorem 1.37]{CLOT}.
\end{proof}

%\begin{remark}
%If $n=2^{t-1}$ then ${\rm zl}_{\Z_2}(\mathbb{C}P^{n}) = 2n-1$, and if $n_1=2^{t_1-1}$ then ${\rm zl}_{\Z_2}(\mathbb{R}P^{n_1})+1 = \TC(\mathbb{R} P^{n_1})=2n_1$. In addition, if $k=0$ then we get $\TC(X(\mathbb{C} P^n, n_1,\dots,n_r)) = \TC(\mathbb{R} P^{n_1}) + r  + 2n.$

%\end{remark}

\section{Equivariant LS-category and topological complexity of a class of $\Z_2$-spaces}\label{sec: eq Cat TC some gen Dold}
Let $N$ be a topological space and $\sigma$ be a fixed point free involution on $N$. Recall the involution $\tau_j$ on $S^{n_j}$ from \eqref{eq: invo prodsphere} for $j=1, \ldots, r$. Define an involution $${\tau' \colon S^{n_1}\times \cdots \times S^{n_r}\times N \to S^{n_1}\times \cdots \times S^{n_r}\times N}$$ by
\begin{equation}\label{eq:free_inv}
\tau'({\bf x}_1, \dots, {\bf x}_{r},y)=(\tau_1({\bf x}_1), \dots, \tau_r({\bf x}_r), \sigma(y)).
\end{equation}
 Then $\tau'$ is a free involution on $S^{n_1}\times \cdots \times S^{n_r}\times N$ and $\Z_2$ acts freely on $S^{n_1}\times \cdots \times S^{n_r}\times N$ via $\tau'$. We denote the orbit space by $X((n_1,p_1) \ldots, (n_r,p_r), N)$.
 %Let $N/\sigma$ is the orbit space $N/\left<\tau \right>$.
 In this section, we study equivariant LS-category and equivariant topologycal complexity of several $\Z_2$-spaces related to the generalized projective product spaces and the generalized Dold manifolds. %{\color{red} I think here we can consider more general involutions $\tau_i$'s (added), also we can add some more concrete computation.}
%We adhere the notation from Section \ref{}. 
\begin{proposition}\label{prop:eqi_lstc}
Let $N$ be a metrizable space with a free involution $\sigma$ and $1 \leq p_i \leq n_i$ for $i=1, \ldots, r$. Then, for the $\Z_2$-action determined by \eqref{eq:free_inv}, we have  
\begin{equation}\label{eq:equi_cat_1}
{\rm cl}_{\Z_2}(N/\left<\sigma \right>)+r+1\leq \mathrm{cat}_{\Z_2}(S^{n_1}\times \cdots \times S^{n_r}\times N)\leq \mathrm{cat}(N/\left<\sigma \right>)+r.
\end{equation}
In particular, if $N=\Sigma_g$ and $\sigma$ is the antipodal involution on $\Sigma_g$ considered in the paragrapgh before Corollary \ref{cor:tc_sigmag}, then $\mathrm{cat}_{\Z_2}(S^{n_1}\times \cdots \times S^{n_r}\times \Sigma_g))=r+3$. 
\end{proposition}
\begin{proof}
Since $N$ is metrizable, the space $S^{n_1}\times \cdots \times S^{n_r}\times N$ is also metrizable. We also have a free action of $\Z_2$ on $S^{n_1}\times \cdots \times S^{n_r}\times N$. Therefore  \cite[Theorem 1.15]{eqlscategory} gives \[\mathrm{cat}_{\Z_2}(S^{n_1}\times \cdots \times S^{n_r}\times N)=\mathrm{cat}(X((n_1,p_1) \ldots, (n_r,p_r), N)).\]
Then the inequalities in \eqref{eq:equi_cat_1} follow from \Cref{Prop_cat}. 

Now assume that $N=\Sigma_g$. We know that the cup-length of $H^*(N_{g+1};\Z_2)$ is $2$. Moreover, $\mathrm{cat}(N_{g+1})=3$. Therefore, 
\[2+r+1\leq \mathrm{cat}_{\Z_2}(S^{n_1}\times \cdots \times S^{n_r}\times \Sigma_g))\leq 3+r.\] Hence the proposition follows.
\end{proof}

\begin{remark}
Let $m_1\leq \dots\leq m_{\ell}$ and  $N=S^{m_1}\times \dots \times S^{m_{\ell}}$. Consider the involution $\sigma$ generated by the antipodal on each sphere $S^{m_j}$ for $1\leq j\leq k$. Then $N/\left<\sigma \right>=P(m_1,\dots,m_{\ell})$. Recall that the cup-length of $P(m_1,\dots,m_{\ell})$ is $m_1+\ell-1$ and $\ct(P(m_1,\dots,m_{\ell}))=m_1+\ell$.  Therefore, $\ct_{\Z_2}(S^{n_1}\times \cdots \times S^{n_r}\times N))=m_1+\ell+r$.
\end{remark}

\begin{proposition}\label{eq tc}
 If $N$ is a smooth manifold with a free involution $\sigma$ and  $n_i\geq 2$ for $i=1, \ldots, r$. Then, for the $\Z_2$-action determined by \eqref{eq:free_inv}, we have the following. 
\begin{equation}\label{eq: TCpi0}
  r+k+{\rm zl}_{\Q}(N)+1\leq\TC_{\Z_2}(S^{n_1}\times \cdots \times S^{n_r}\times N)\leq r+k+\TC_{\Z_2}(N),   
 \end{equation}
  where $k$ is the number of even $n_i$'s and $p_i=0$ for $1\leq i\leq r$. Moreover,  
  \begin{equation}\label{eq: TCpi2}
   r+k+{\rm zl}_{\Q}(N)+1\leq\TC_{\Z_2}(S^{n_1}\times \cdots \times S^{n_r}\times N)\leq 2r+\TC_{\Z_2}(N),   
  \end{equation}
  when $p_i\geq 2$ for $1\leq i\leq r$.
\end{proposition}
\begin{proof}
Observe that \[H^*(S^{n_1}\times \cdots \times S^{n_r}\times N);\Q) \cong \Lambda_{\Q}(\alpha_1,\dots,\alpha_{r-k})\otimes \otimes_{j=1}^{k}\displaystyle\frac{\Q[\beta_{j}]}{\left<\beta_j^2\right>}\otimes H^*(N;\Q).\]
Let $\bar{\alpha}_i:=\alpha_i\otimes 1-1\otimes\alpha_i$, for $1\leq i\leq r-k$. Let $\bar{\beta}_j:=\beta_j\otimes 1-1\otimes\beta_j$, for $1\leq j\leq k$. Suppose that ${\rm zl}_{\Q}(N)$ is the zero-divisor cup-length of $H^*(N; \Q)$. That is, there are elements $X_1,\dots, X_{\ell}$ in the kernel of $\triangle_N^* \colon H^*(N\times N)\to H^*(N)$, where $\triangle_N(x)=(x,x)$ such that $\prod_{i=1}^{\ell}X_i\neq 0$. Then note that the product \[\prod_{i=1}^{r-k}\bar{\alpha}_{i}\cdot\prod_{j=1}^{k}\bar{\beta}_{j}^{2}\cdot \prod_{s=1}^{\ell}X_s\neq 0.\]  Therefore, $r+k+ {\rm zl}_{\Q}(N) \leq \TC_{\Z_2}(S^{n_1}\times \cdots \times S^{n_r}\times N) $. Consider the case $p_i=0$ for $i=1,\dots, r$. Then it follows from  \cite[Lemma 4.1]{EqTCGrant}, 
\[\TC_{\Z_2}(S^n)=
\begin{cases}
 2& \text{ if $n$ is odd,} \\
 3& \text{ if $n$ is even}.
\end{cases}
\]
Therefore, using generalized additive inequality of topological complexity \cite[Theorem 4.2]{EqTCGrant}, we get
\begin{align*}
\TC_{\Z_2}(S^{n_1}\times \cdots \times S^{n_r}\times N) &\leq \TC_{\Z_2}(S^{n_1}\times \cdots \times S^{n_r})+ \TC_{\Z_2}(N)-1\\
& \leq 2(r-k)+3k-r+1+\TC_{\Z_2}(N)-1\\
&= r+k +\TC_{\Z_2}(N).
\end{align*}
Now assume that $p_i\geq 2$. So the fixed point set of the $\Z_2$-action on $S^{n_1}\times \cdots \times S^{n_r}$ is path-connected. Then 
\begin{align*}
\TC_{\Z_2}(S^{n_1}\times \cdots \times S^{n_r}\times N) &\leq \TC_{\Z_2}(S^{n_1}\times \cdots \times S^{n_r})+ \TC_{\Z_2}(N)-1\\
& \leq 2\ct_{\Z_2}(S^{n_1}\times \cdots \times S^{n_r})-1+\TC_{\Z_2}(N)-1\\
&= 2r +\TC_{\Z_2}(N),
\end{align*}
since $\ct_{\Z_2}(\prod_{i=1}^{r}S^{n_i})\leq r+1$ when $p_i\geq 2$.
This proves the proposition.
\end{proof}

\begin{remark} We observe the following.
\begin{enumerate}
    \item If the $\Z_2$-equivariant topological complexity of $N$ coincides with the zero-divisor cup-length of $N$ plus one, then we have the equality in \eqref{eq: TCpi0}. 
    \item If $k=r$ and $\Z_2$-equivariant topological complexity of $N$ coincides with the zero-divisor cup-length of $N$ plus one, then we have the equality in \eqref{eq: TCpi2}.
    %Let $2\geq m_1\leq \dots\leq m_{\ell}$ and  $N=S^{m_1}\times \dots \times S^{m_{\ell}}$. Then \[\TC_{\Z_2}(N)\leq 2\ct_{\Z_2}(N)-1=2(\ell+1)-1.\]  
\end{enumerate}
\end{remark}

Let $M$ be a topological space and $\tau$ an involution on $M$ with non-empty fixed point set. 
Define an involution $\sigma' \colon M\times S^{n_1}\times \cdots \times S^{n_r} \to M\times S^{n_1}\times \cdots \times S^{n_r}$ 
by \[\sigma'(y,{\bf x}_1,\dots,{\bf x}_r)=(\tau(y),-{\bf x}_1,\dots,-{\bf x}_r).\] Then $\Z_2$ acts freely on $M\times S^{n_1}\times \cdots \times S^{n_r}$ via $\sigma'$. We denote the orbit space by $X(M, n_1, \ldots, n_r)$.

\begin{proposition} Let $M$ be a compact, simply connected and path-connected metrizable space with an involution $\tau$ such that $\tau^*$ is identity. Then 
\begin{equation}\label{eq:equi_cat2}
n_1+r + {\rm cl}_{\Z_2}(M)\leq \mathrm{cat}_{\Z_2}(M\times S^{n_1}\times \cdots \times S^{n_r})\leq n_1+r+\mathrm{cat}_{\Z_2}(M)-1,
\end{equation}
\end{proposition}
\begin{proof}
Since $M\times S^{n_1}\times \cdots \times S^{n_r}$ is a metrizable space and $\sigma'$ acts freely on it, then it follows from \cite[Theorem 1.15]{eqlscategory} that $$\mathrm{cat}_{\Z_2}(M\times S^{n_1}\times \cdots \times S^{n_r})=\mathrm{cat}(X(M,n_1,\dots,n_r)).$$ Thus, the right inequality  of \eqref{eq:equi_cat2} follows from the fact that $\mathrm{cat}(P(n_1,\dots,n_r))=n_1+r$ and \Cref{prop cat}. 

One can compute the cup-length of $H^*(X(M,n_1,\dots,n_r);\Z_2)$ using \Cref{thm:toric_mod2}. Observe that this cup-length is $n_1-1+r+{\rm cl}_{\Z_2}(M)$. So, the left inequality of \eqref{eq:equi_cat2} follows. 
\end{proof}

\begin{example}\label{ex:eqi_lstc}
Let $M=\mathbb{C}P^{n}$ with the complex conjugation involution. One can show that $\mathrm{cat}_{\Z_2}(\mathbb{C}P^n)=n+1$, see the proof of \cite[Proposition 3.10 ]{Naskar}. Therefore,   $\TC_{\Z_2}(\mathbb{C}P^n)\leq 2n+1$. Moreover, ${\rm zl}_{\Q}(\mathbb{C}P^n)=2n$. Now using \Cref{eq tc}, we get
\[\TC_{\Z_2}(\mathbb{C} P^n\times  S^{n_1}\times \cdots \times S^{n_r})=r+k+2n+1,\] where 
$k$ be the number of even $n_i$'s. More generally, for $M=\Gr_d(\C^n)$ and $\tau$ the complex conjugation involution on $\Gr_d(\C^n)$, we have
\[2d(n-d)+r+k+1 \leq \TC_{\Z_2}(\Gr_d(\C^n)\times S^{n_1}\times \cdots \times S^{n_r})\leq 2d(n-d)+r+k+1,\] where $k$ is the number of even $n_i$'s. That is $\TC_{\Z_2}(\Gr_d(\C^n)\times S^{n_1}\times \cdots \times S^{n_r})=2d(n-d)+r+k+1$.
\end{example}

\noindent {\bf Acknowledgement}: The authors thank Donald M. Davis, Mark Grant and Stephan Mescher for some helpful discussion. The first author thank IIT Bombay and the second author thank ICSR of IIT Madras.

\bibliographystyle{plain} 
\bibliography{references}

\end{document}